\documentclass{amsart}
\usepackage{fullpage}
\usepackage{amsfonts,enumerate}
\usepackage[mathscr]{euscript}
\usepackage{epsfig}
\usepackage{latexsym}
\usepackage[all]{xy}
\usepackage{amssymb}
\usepackage{amsthm}
\usepackage{amsmath}
\usepackage{amsxtra}
\usepackage[dvipsnames]{xcolor}
\usepackage[colorlinks]{hyperref}
\hypersetup{citecolor=blue}
\usepackage{tikz, tikz-cd, etoolbox}
\usetikzlibrary{arrows}
\usetikzlibrary{calc}
\usetikzlibrary{decorations.markings}
\newcommand{\zerodisplayskips}{%
  \setlength{\abovedisplayskip}{4pt}%
  \setlength{\belowdisplayskip}{2pt}%
  \setlength{\abovedisplayshortskip}{0pt}%
  \setlength{\belowdisplayshortskip}{0pt}}
\appto{\normalsize}{\zerodisplayskips}
\appto{\small}{\zerodisplayskips}
\appto{\footnotesize}{\zerodisplayskips}
\usepackage{algorithmic}
\usepackage{float}
\usepackage{subcaption}

.tfm
.tfm

\theoremstyle{plain}
\newtheorem{prop}{Proposition}[section]
\newtheorem{alg}[prop]{Algorithm}
\newtheorem{thm}[prop]{Theorem}
\newtheorem{coro}[prop]{Corollary}

\newtheorem{lemma}[prop]{Lemma}

\newtheorem{problem}[prop]{Problem}
\theoremstyle{definition}
\newtheorem{defi}[prop]{Definition}

\newtheorem{exam}[prop]{Example}

\theoremstyle{remark}

\newtheorem{remark}[prop]{Remark}

\numberwithin{table}{section}

\DeclareMathOperator{\num}{Num}
\DeclareMathOperator{\den}{Den}
\DeclareMathOperator{\Two}{Two}
\DeclareMathOperator{\bd}{bd}
\DeclareMathOperator{\Bi}{Bv}
\DeclareMathOperator{\Ext}{Ext}

\DeclareMathOperator{\Res}{Res}

\DeclareMathOperator{\End}{End}

\DeclareMathOperator{\Aut}{Aut}

\DeclareMathOperator{\val}{val}

\DeclareMathOperator{\quo}{quo}

\newcommand{\Poincare}{\mathcal H}
\newcommand{\Perm}{P}
\newcommand{\calH}{\mathcal H}
\newcommand{\calD}{\mathcal D}
\newcommand{\calC}{\mathcal C}

\newcommand{\calP}{\mathcal P}

\newcommand{\F}{\mathbb F}

\newcommand{\GL}{{\rm GL}}
\newcommand{\SL}{{\rm SL}}
\newcommand{\PGL}{{\rm PGL}}
\newcommand{\PSL}{{\rm PSL}}

\newcommand{\Sub}{\Delta}
\newcommand{\Kul}{\mathcal K}
\newcommand{\Tree}{\mathfrak T}

\newcommand{\Mod}[1]{\ (\mathrm{mod}\ #1)}
\makeatletter
\newcommand{\verbatimfont}[1]{\renewcommand{\verbatim@font}{\ttfamily#1}}
\makeatother

\def\TT{\mathcal T}
\def\ZZ{\mathbb Z}
\def\SS{\mathbb S}

\def\FF{\mathcal F}

\def\<#1>{{\left\langle{#1}\right\rangle}}

\def\Z{{\mathbb Z}}             

\begin{document}

\title{On computing finite index subgroups of $\PSL_2(\Z)$}

\author{Nicol\'as Mayorga Uruburu}
\address[N. Mayorga]{FAMAF-CIEM, Universidad Nacional de
  C\'ordoba. C.P:5000, C\'ordoba, Argentina.}
\email{nmayorgau@unc.edu.ar}

\author{Ariel Pacetti}
\address[A. Pacetti]{Center for Research and Development in Mathematics and Applications (CIDMA),
	Department of Mathematics, University of Aveiro, 3810-193 Aveiro, Portugal}
\email{apacetti@ua.pt}

\author{Leandro Vendramin}
\address[L. Vendramin]{
Department of Mathematics, Vrije Universiteit Brussel, Pleinlaan 2, 1050 Brussel, Belgium}
\email{Leandro.Vendramin@vub.be}

\thanks{N.M.U. was partially supported by a Conicet grant, AP was
  partially supported by the Portuguese Foundation for Science and
  Technology (FCT) within project UIDB/04106/2020 (CIDMA), L.V. was
  supported by project OZR3762 of Vrije Universiteit Brussel}


\begin{abstract}
We present a method to compute finite index subgroups of $\PSL_2(\Z)$. Our strategy follows Kulkarni’s ideas, the main contribution being a recursive method to compute bivalent trees as well as their automorphism group. As a concrete application, we compute all subgroups of index up to $20$. We then use this database to produce tables with several arithmetical properties. 
\end{abstract}

\maketitle

\section{Introduction} 
Let $\PSL_2(\mathbb{Z})$ be the modular group, obtained as the
quotient of the set of all two by two matrices with integral entries
of determinant $1$ modulo the subgroup
$\{\pm\left(\begin{smallmatrix}1 & 0\\ 0 &
    1\end{smallmatrix}\right)\}$. It is a classical problem that of
determining all subgroups of $\PSL_2(\ZZ)$ of a given finite index. In
a remarkable article Newman (\cite{MR466047}) computed the number of
subgroups of index up to a hundred (including the asymptotic behavior
of the counting function). The group $\SL_2(\ZZ)$ acts on the set of
subgroups by conjugation, and it is also a natural problem that of
determining the subgroups (or their number) up to
$\SL_2(\ZZ)$-equivalence. Similarly, the group $\GL_2(\ZZ)$ acts on
$\PSL_2(\ZZ)$ so the same sort of questions for 
$\GL_2(\ZZ)$-equivalence classes can be considered. In \cite{math/0702223} the author
gave a method to compute the number of $\SL_2(\ZZ)$-equivalence
classes, whose sequence corresponds to the integer sequence
\emph{A121350} (see \url{https://oeis.org/A121350}).

From a computational point of view, it is challenging to actually
compute tables of subgroups of $\PSL_2(\ZZ)$ (modulo conjugation). To
our knowledge, it was Kulkarni (in \cite{Kulkarni1991}) the first one
who actually proposed an algorithm to compute them. In the
aforementioned article (in Appendix 1) he even gives a list of
subgroups of index up to $6$ (for these small indexes, there are only
a few of them). Let us give some details on Kulkarni's approach, which
also shows the relevance of the problem.

Let
$\calH^\ast=\{z\in\mathbb{C}\mid\operatorname{Im}z>0\}\cup\{\infty\}\cup\mathbb{Q}$
denote the extended complex upper half plane.  The group $\SL_2(\Z)$
acts on $\calH^\ast$ by M\"oebius transformations. Since the matrix
$\left(\begin{smallmatrix}-1 & 0\\ 0 & -1\end{smallmatrix}\right)$
acts trivially, we get a well defined action of the modular group
$\PSL_2(\ZZ)$. The typical fundamental domain $\mathcal{T}$ for the
action of $\PSL_2(\ZZ)$ on $\calH^\ast$ is the hyperbolic triangle with
vertices $\rho= \exp(2 \pi i /6)$, $\rho^2$ and $\infty$. The
half-plane $\calH$ has an extra transformation $\iota$ (as a real
variety) given by $\iota(z)= -\overline{z}$. This allows to extend the
action of $\PSL_2(\Z)$ on $\calH^\ast$ to an action of $\PGL_2(\Z)$ on
$\calH^\ast$ given by
\[
  \begin{pmatrix}
    a & b\\
    c & d
  \end{pmatrix} \cdot z =
  \begin{cases}
    \frac{az+b}{cz+d} & \text{ if }ad-bc = 1,\\
    \frac{a\overline{z}+b}{c\overline{z}+d} & \text{ if }ad-bc = -1.
\end{cases}
\]
The hyperbolic triangle $\FF$ with vertices $\rho$, $i$ and $\infty$
(see Figure~\ref{tria}) is a fundamental domain for the action of
$\PGL_2(\ZZ)$ on $\calH^\ast$. Clearly $\TT = \FF \cup \iota(\FF)$.
\begin{figure}[ht]
   \centering
   \begin{tikzpicture}
   \filldraw[draw=black,fill=gray!60] plot [domain=0:1] (\x,{sqrt(4-\x^2)}) node[label={-60:$\rho$},inner sep=1.5pt,circle,fill=blue]{} -- plot [smooth,domain=1:0] (\x,4) -- cycle;
   \draw (3,0) to (-1,0)  node[inner sep=0pt]{};
   \draw (0,3) to (0,-1)  node[inner sep=0pt]{};
  \draw (0,2) node[label={left:$i$},inner sep=1.5pt,circle,fill=red]{};
  \draw (0.5,3) node[inner sep=0pt]{$\mathcal{F}$};
  \draw (0.5,1.3) node[label={above:$f$},inner sep=1.5pt]{};
  \end{tikzpicture}
  \caption{Fundamental domain for $\PGL_2(\Z)$.}
  \label{tria}
\end{figure}
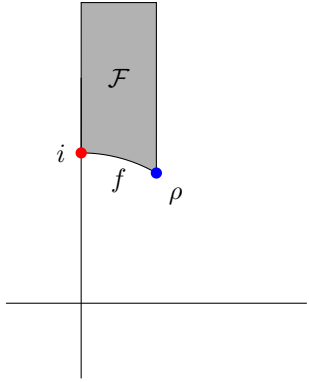
Following \cite{Kulkarni1991} in the present article we use the following terminology.

\begin{defi}
  The points of $\calH$ in the orbit of $i$ (resp. in the orbit of
  $\rho$) by the action of $\GL_2(\Z)$ are called \emph{even or red
    vertices} (resp. \emph{odd or blue vertices}).  The points in the orbit of
  $\SL_2(\Z) \cdot\infty$ are called \emph{cusps}.
\end{defi}
The $\PGL_2(\ZZ)$-translates of the triangle $\FF$ provide a tessellation of
$\calH$.  The hyperbolic geodesic of $\FF$ joining $i$ with $\infty$
(resp. joining $\rho$ with $\infty$) is called an \emph{even edge}
(resp. \emph{odd edge}) and so is all of its translates by
$\GL_2(\Z)$. The hyperbolic geodesic of $\FF$ joining $i$ with $\rho$
is of finite length. Its $\GL_2(\Z)$-translates are called
\emph{$f$-edges}.

Let $\Sub$ be a finite index subgroup of $\PSL_2(\ZZ)$. Then a fundamental
domain for its action is given by the union of translates of $\TT$. A
key idea introduced by Kulkarni in \cite{Kulkarni1991} was to replace
the fundamental domain $\TT$ by
\begin{equation}
  \label{eq:fund-triang}
\calD = \FF \cup \iota(S\cdot \FF),  
\end{equation}
where
$S = \left(\begin{smallmatrix} 0 & -1\\ 1 &
    0\end{smallmatrix}\right)$. In this way, any fundamental domain
obtained as the union of translates of $\calD$ will have no $f$-edges
on its border. Furthermore, if the fundamental domain has nice
properties, the set of $f$-edges inside the fundamental domain form a
graph with an orientation on it. It is expected that the subgroup
$\Sub$ could be recover (up to conjugation) from the graph. This
allowed Kulkarni to give three different ways to describe a finite
index subgroup:
\begin{enumerate}
\item Via the graph of $f$-edges on the quotient $\Sub \backslash \Poincare$, what Kulkarni called a bipartite cuboid graph.
  
\item Via removing from the whole graph of $f$-edges those vertices
  having valence $2$ and making some cuts to the graph in order to get
  a tree, what Kulkarni called a tree diagram.
  
\item Via listing the cusps appearing on a ``special'' fundamental
  domain, and listing how the paths going throw them are glued
  together. This was called a Farey symbol by Kulkarni.
\end{enumerate}

For theoretical purposes, the first approach is the best one, as on
the one hand it contains all the information needed to recover the
subgroup and on the other one it is in bijection with the set
conjugacy classes of subgroups. However, the second one is better for
computational purposes (which is the goal of the present article)
since the tree has a smaller number of edges.

In \cite[\S3.1]{MR0337781} the authors gave an alternative method to
describe a finite-index subgroup of $\PSL_2(\ZZ)$, following the ideas
introduced by Millington in \cite{Millington}. Roughly speaking, if
$\Sub$ is a subgroup of $\PSL_2(\ZZ)$ of index $d$, then $\PSL_2(\ZZ)$
acts by left multiplication on right coset representatives of
$\Sub \backslash\PSL_2(\ZZ)$ (in particular the action of each element of
$\PSL_2(\ZZ)$ is given by a permutation of $\SS_d$). To describe the
action completely it is enough to determine the permutations attached
to a set of generators of $\PSL_2(\ZZ)$, for example the permutation
attached the matrix $S$ introduced before and the one attached to the
matrix
$T=\left(\begin{smallmatrix} 1 & 1\\ 0 & 1\end{smallmatrix}\right)$
(the usual translation matrix). This way of representing a subgroup is
also called a ``passport'' in the literature.  The method was extended
in \cite{Stromberg2019}, where a table containing the permutations
attached to all subgroups (up to $\GL_2(\ZZ)$-equivalence) of index up
to $12$ (and up to index $18$ in electronic format) is
given. Stromberg's computations were extended in different ways
(including computing tables of modular forms) in \cite{2207.13365} and
\cite{2301.02135}.

The purpose of the present article is to provide an algorithm to
compute finite index subgroups (up to $\SL_2(\ZZ)$ and $\GL_2(\ZZ)$
equivalence) via computing tree diagrams. A \emph{bi-valent} tree is a
tree whose set of valences has two elements.  Our main contribution
(of interest on its own) is to provide a recursive algorithm to
construct bi-valent trees. Along the way we prove interesting
properties of the automorphism group of a bi-valent tree (see
Theorem~\ref{thm:auto}) which is crucial to speed up the algorithms
used to computed tree diagrams.

Since our final goal is to provide tables of finite index subgroups of
$\PSL_2(\ZZ)$ (up to both $\SL_2(\ZZ)$ and $\GL_2(\ZZ)$ equivalence)
we also provide the needed algorithms to construct the subgroup
attached to a tree diagram (as well as its passport
representation). Most of these algorithms are already part of the
literature.  Our humble contribution is to include missing details
(that appeared while writing the code) regarding ``boundary issues''
as well as correcting a few mistakes. Of particular interest is the
algorithm to construct the generalized Farey symbol attached to a tree
diagram (see Theorem~\ref{thm:tree-diagram2Farey-symbol}). Although
such an algorithm plays a crucial role in \cite{Kulkarni1991}, somehow
there is no description of such an algorithm in the article, so we take the opportunity to present it here, and prove its correctness.

Our algorithms are implemented in the computer software
\textsf{GAP} \cite{GAP4}. The
code can be downloaded from 
\url{https://github.com/vendramin/subgroups} with DOI \textsf{10.5281/zenodo.8113635}. 
The GitHub repository contains some precomputed data as well as a tutorial
(with examples) on how to use the code. The package depends on the packages:
\textsf{Yags} for graph theory, \cite{MR2735087} for some calculations related
to generalized Farey symbols, and \textsf{HomAlgTools} for speeding up
recursive functions.

The article is organized as follows: Section~\ref{sec:kulkarni}
contains a quick review of Kulkarni's main results used in the
present article. It includes some definitions (used during the
article) as well as the correspondence between subgroups and some
particular fundamental domains (called special polygons). This result
justifies our computation of tree
diagrams. Section~\ref{sec:bi-valent} contains the main recursive
algorithm to compute bi-valent trees as well as their automorphism
group. Although our trees have valence $\{1,3\}$, the proven results hold for
any tree with valence set $\{1,n\}$, $n$ being arbitrary. Of
particular interest is Theorem~\ref{thm:auto}, which describes the
automorphism group of a bi-valent tree as a semi-direct product of two
other ones. The section includes algorithms used to add an orientation
and a coloring on the set of external vertices to a bi-valent tree.

Section~\ref{section:treetosubgroup} contains the definition of a
generalized Farey symbol (g.F.s. for short), which consists of the set
of cusps of a special polygon. As previously mentioned, a given
g.F.s. together with the gluing of the sides allow to give an
alternative description of a finite index subgroup of
$\PSL_2(\Z)$. This section contains an algorithm to compute a
g.F.s. attached to a tree diagram
(Theorem~\ref{thm:tree-diagram2Farey-symbol}), providing a way to go
from the second description to the third one. A tree diagram together
with a g.F.s. is called a \emph{Kulkarni diagram}. The same section
contains different algorithms to, given a Kulkarni diagram
corresponding to a subgroup $\Sub$, compute arithmetic information of
the quotient curve $\Sub\backslash \calH^\ast$ (like cusp
representatives and width of the cusps, genus, ramification points,
etc) as well as information of the subgroup $\Sub$ itself (including a
reduction algorithm needed to solve the word problem, and finding a
set of coset representatives).

Section~\ref{sec:passports} contains the definition of a passport and
algorithms to compute the passport attached to a Kulkarni diagram (to
compare with other tables in the literature). At last,
Section~\ref{sec:data} contains some tables and interesting facts that
can be deduced from the data we computed.

It should be clear to the reader how Kulkarni's article crucially
influenced our work, so we strongly recommend they to look at the
article \cite{Kulkarni1991} which contains more details on the
correspondences between tree diagrams, Farey symbols and bipartite
cuboid graphs.

\section{A quick guide to Kulkarni's approach}
\label{sec:kulkarni}
Recall the following definition from \cite[\S4.1]{Kulkarni1991}.

\begin{defi}
  A \emph{bipartite cuboid graph} is a finite graph whose vertex set
  is divided into two disjoint subsets $V_0$ (the red ones) and $V_1$
  (the blue ones) such that
  \begin{enumerate}
  \item every vertex in $V_0$ has valence $1$ or $2$,
    
  \item every vertex in $V_1$ has valence $1$ or $3$,
    
  \item there is a prescribed cyclic order on the edges incident at each vertex of valence $3$ in $V_1$,
    
  \item every edge joins a vertex in $V_0$ with a vertex in $V_1$.
  \end{enumerate}
\end{defi}

Let $\Sub$ be a finite index subgroup of $\PSL_2(\ZZ)$. The quotient
$\calH^\ast \backslash \Sub$ is a compact oriented real surface, with a
tessellation given by the translates of $\calD$. The graph of
$f$-edges of the surface is an example of a cuboid graph, where $V_0$
corresponds to the red vertices (i.e. the translates of $i$) while
$V_1$ corresponds to the blue vertices (the translates of $\rho$).

We will only use bipartite cuboid graphs twice on the present article:
once in Example~\ref{exm:isom}, to prove that two non-isomorphic tree
diagrams provide conjugated subgroups and also while computing the
generalized Farey symbol attached to a tree diagram. In contrast to a
bipartite cuboid graph, the notion of a tree diagram appears while
working with a fundamental domain (a \emph{special polygon}) for the
quotient. The advantage of this approach is that
we do not need to glue its edges together, it is enough to store
which/how edges get identified. In this way we can work with a tree
instead of a graph.

\begin{defi}
  A \emph{special polygon} is a convex hyperbolic polygon $\calP$ whose boundary $\partial \calP$ is a union of even and odd edges, satisfying the following properties:
  \begin{enumerate}
  \item The even edges in $\partial \calP$ come in pairs, each pair forming
    a complete hyperbolic geodesic called \emph{even line}.
    
  \item The odd edges in $\partial \calP$ come in pairs. The edges in each pair meet at an odd vertex, making and internal angle of $\frac{2\pi}{3}$.
    
  \item There exists an involution on the edges of $\partial \calP$ so that no edge is carried into itself.
    
  \item The involution sends an odd edge into another odd edge, making an internal angle of $\frac{2 \pi}{3}$ between themselves.
    
  \item Let $e_1,e_2$ be two even edges in $\partial \calP$ forming a even line. Then
    either $e_1$ is paired to $e_2$, or else $\{e_1,e_2\}$ form a
      \emph{free side} of $\partial \calP$ and this free side is paired to
      another such free side of $\partial \calP$.
    
  \item $0$ and $\infty$ are two of the vertices of $\calP$.
  \end{enumerate}
\end{defi}

\begin{remark}
  \label{remark:orientation}
  Note that any special polygon has a canonical orientation induced
  from the ``counterclockwise'' orientation of $\calH$. The action of
  $\SL_2(\Z)$ on a special polygon preserves the orientation (since it
  corresponds to a holomorphic function). However, the action of the
  matrix
  $\left(\begin{smallmatrix} -1 & 0\\ 0 & 1\end{smallmatrix}\right)$
  has the effect of reversing the orientation of our special
  polygon. The choice of the opposite orientation corresponds to the
  action by $\iota$ on $\calH$. The choice of the counterclockwise
  orientation or its inverse accounts for the difference between
  equivalence classes of subgroups up to $\SL_2(\Z)$ conjugation or up
  to $\GL_2(\Z)$ conjugation. This will prove crucial later (see
  Remark~\ref{exm:orientation2}).
\end{remark}

\begin{thm}[Kulkarni]
  A special polygon is a fundamental domain for the subgroup $\Sub$
  generated by the side-pairing transformations and these
  transformations form an independent set of generators for
  $\Sub$. Conversely every subgroup $\Sub$ of finite index in $\PSL_2(\Z)$
  admits a special polygon as a fundamental domain.
\end{thm}

\begin{proof}
	See page 1055 of \cite{Kulkarni1991}.
\end{proof}

A special polygon is the union of translates of $\calD$, so as
explained before, we can look at its tree of $f$-edges. Properties
$(4)$ and $(5)$ of a special polygon imply that the end points of the
$f$-edges tree will be points that are one of:
\begin{itemize}
\item the intersection of two odd edges, or
  
\item the intersection of two even edges paired together, or
  
\item the intersection of two even edges forming a free side.

\end{itemize}

Any vertex of the tree has valence $1$ (corresponding to an end
point), $2$ or $3$. To avoid redundant information on the tree, one
can remove the vertices having valence $2$, obtaining the so called
\emph{tree diagrams}. A great advantage of the tree diagram is that
its number of vertices tends to be much smaller than that of the
original bipartite cuboid graph. Remark~\ref{remark:orientation}
implies that the vertices with valence $3$ have a natural orientation.

\begin{defi}
  \label{defi:tree-diagram}
  A \emph{tree diagram} is a tree with at least one edge such that
  \begin{enumerate}
  \item all internal vertices are of valence $3$,
    
  \item there is a prescribed cyclic order on the edges incident at each internal vertex (orientation),
    
  \item the terminal vertices are partitioned into two possibly empty subsets $R$ and $B$ (red and blue vertices),
    
  \item there is an involution $\iota$ on $R$.
  \end{enumerate}
\end{defi}
By $\Bi(n,3)$ we will denote the set of tree diagrams with $n$ internal vertices.

\begin{exam}
  \label{exm:orientation2}
  It is not true in general that if $\Tree$ is a tree diagram, and we
  consider the tree diagram where the orientation at all internal
  vertices are inverted (by considering the inverse of the
  permutation) we get isomorphic tree diagrams. For example, the
  following planar (with the anti-clockwise orientation) tree diagrams
  in $\Bi(2,3)$ are not isomorphic. They correspond to
  two subgroups of $\PSL_2(\Z)$ which are not isomorphic via
  conjugation under $\SL_2(\Z)$, but are isomorphic under conjugation
  by $\GL_2(\Z)$.
  \begin{figure}[H]
   \centering
   \begin{tikzpicture}
     \draw (0,0) node[inner sep=0pt,draw=blue,circle,minimum size=3pt,fill=blue]{} -- (0.5,0.5) -- (1.5,0.5) -- (2,0) node[inner sep=0pt,draw=red,circle,minimum size=3pt,fill=red]{};
     \draw (0,1) node[inner sep=0pt,draw=red,circle,minimum size=3pt,fill=red]{} -- (0.5,0.5) -- (1.5,0.5) -- (2,1) node[inner sep=0pt,draw=red,circle,minimum size=3pt,fill=red]{};
   \end{tikzpicture}\qquad\qquad   
   \begin{tikzpicture}
     \draw (0,0) node[inner sep=0pt,draw=red,circle,minimum size=3pt,fill=red]{} -- (0.5,0.5) -- (1.5,0.5) -- (2,0) node[inner sep=0pt,draw=red,circle,minimum size=3pt,fill=red]{};
     \draw (0,1) node[inner sep=0pt,draw=blue,circle,minimum size=3pt,fill=blue]{} -- (0.5,0.5) -- (1.5,0.5) -- (2,1) node[inner sep=0pt,draw=red,circle,minimum size=3pt,fill=red]{};
   \end{tikzpicture}
\end{figure}    
\end{exam}

The way to relate a tree diagram with a bipartite cuboid graph as
follows: remove from the bipartite cuboid graphs all vertices of
valence $2$ (connecting the remaining endings) and make some cuts to
the graph so that it becomes a tree (see \cite[\S4.4]{Kulkarni1991} 
for details). Conversely, given a tree diagram, color blue 
all of its internal vertices (i.e. they are odd vertices) and keep the
coloring of the tree diagram on the external vertices. If two
consecutive vertices of the tree are blue, enlarge the tree by adding
one extra vertex and paint it with red color.  Finally, we need to
glue together the identified external vertices: identify two external
vertices $v$ and $w$ in $R$ if $\iota(v) = w$, where $\iota$ is the
involution of $R$. It is not hard to verify that the resulting graph
is a bipartite cuboid graph.

\vspace{1pt}

\begin{remark}
  Following Kulkarni's notation, the way we encode the involution
  $\iota$ is as follows: if $v$ is a vertex fixed by $\iota$ then
  we just save its color. Otherwise, there exists $w \neq v$ such that
  $\iota(v)=w$. In this case, we add a label to each of the two
  vertices.
\end{remark}

As will be described in Section~\ref{section:treetosubgroup}, there
  is a finite-to-one surjective map between the set of tree diagrams
  and the set of conjugacy classes of finite index subgroups of
  $\PSL_2(\ZZ)$. This map unfortunately is not injective.
  
\begin{exam}
\label{exm:isom}
  Consider the following tree diagrams with anti-clockwise orientation:
  \begin{figure}[H]
   \centering
   \begin{tikzpicture}
     \draw (0,0) node[inner sep=0pt,draw=red,circle,minimum size=3pt,fill=red]{} -- (0.5,0.5) -- (1.5,0.5) -- (2,0) node[label={-60:$1$},inner sep=0pt,draw=red,circle,minimum size=3pt,fill=white]{};
     \draw (0,1) node[label={120:$1$},inner sep=0pt,draw=red,circle,minimum size=3pt,fill=white]{} -- (0.5,0.5) -- (1.5,0.5) -- (2,1) node[inner sep=0pt,draw=red,circle,minimum size=3pt,fill=red]{};
   \end{tikzpicture}\qquad\qquad
   \begin{tikzpicture}
     \draw (0,0) node[label={-120:$1$},inner sep=0pt,draw=red,circle,minimum size=3pt,fill=white]{} -- (0.5,0.5) -- (1.5,0.5) -- (2,0) node[inner sep=0pt,draw=red,circle,minimum size=3pt,fill=red]{};
     \draw (0,1) node[inner sep=0pt,draw=red,circle,minimum size=3pt,fill=red]{} -- (0.5,0.5) -- (1.5,0.5) -- (2,1) node[label={60:$1$},inner sep=0pt,draw=red,circle,minimum size=3pt,fill=white]{};
   \end{tikzpicture}
\end{figure}
These two tree diagrams are not isomorphic: the right tree satisfies
that if we move from a red vertex (following the orientation) we end
on a free one, while this is not true for the left one. The bipartite
cuboid graph attached to each one of them are the following:
\begin{figure}[H]
   \centering
   \begin{tikzpicture}
     \draw (0,1) node[label={120:$3$},inner sep=0pt,draw=red,circle,minimum size=3pt,fill=white]{} -- (0.5,0.5) -- (1.5,0.5) -- (2,1) node[label={60:$5$},inner sep=0pt,draw=red,circle,minimum size=3pt,fill=red]{};
     \draw (0,0) node[label={-120:$4$},inner sep=0pt,draw=red,circle,minimum size=3pt,fill=red]{} -- (0.5,0.5) node[label={above:$1$},inner sep=0pt,draw=blue,circle,minimum size=3pt,fill=blue]{} -- (1.5,0.5) node[label={above:$2$},inner sep=0pt,draw=blue,circle,minimum size=3pt,fill=blue]{} -- (2,0) node[label={-60:$3$},inner sep=0pt,draw=red,circle,minimum size=3pt,fill=white]{};
     \draw (1,0.5) node[label={above:$6$},inner sep=0pt,draw=red,circle,minimum size=3pt,fill=red]{};
   \end{tikzpicture}\qquad\qquad
   \begin{tikzpicture}
     \draw (0,1) node[label={120:$3'$},inner sep=0pt,draw=red,circle,minimum size=3pt,fill=red]{} -- (0.5,0.5) -- (1.5,0.5) -- (2,1) node[label={60:$4'$},inner sep=0pt,draw=red,circle,minimum size=3pt,fill=white]{};
     \draw (0,0) node[label={-120:$4'$},inner sep=0pt,draw=red,circle,minimum size=3pt,fill=white]{} -- (0.5,0.5) node[label={above:$1'$},inner sep=0pt,draw=blue,circle,minimum size=3pt,fill=blue]{} -- (1.5,0.5) node[label={above:$2'$},inner sep=0pt,draw=blue,circle,minimum size=3pt,fill=blue]{} -- (2,0) node[label={-60:$5'$},inner sep=0pt,draw=red,circle,minimum size=3pt,fill=red]{};
     \draw (1,0.5) node[label={above:$6'$},inner sep=0pt,draw=red,circle,minimum size=3pt,fill=red]{};
   \end{tikzpicture}    
 \end{figure}
\noindent  where the edges labeled with the same name are glued together.  The
 two new graphs are isomorphic under the map
\[\left(\begin{array}{cccccc}
1 & 2 & 3 & 4 & 5 & 6\\ 
2' & 1' & 6' & 5' & 3' & 4'
        \end{array}\right).
\]
It is easy to verify that this map preserves orientations.
\end{exam}
This is the first example where two non-isomorphic tree diagrams give
conjugated subgroups (corresponding to subgroups of index
$6$). However, it seems that for larger index subgroups (when there
are many free sides) the number of tree diagrams is much larger than
the number of $\SL_2(\ZZ)$-classes of subgroups, for example there are
$28$ tree diagrams corresponding to subgroups of order $9$ while only
$14$ $\SL_2(\ZZ)$-equivalence classes. See the data in Table~\ref{table:sizes}.

\begin{remark}
  Working with bipartite cuboid graphs has the advantage that their
  isomorphism classes are in bijection with $\SL_2(\ZZ)$ conjugacy
  classes of subgroups of $\PSL_2(\ZZ)$, but as already mentioned, has
  the disadvantage that it implies working with much larger graphs (which
  are not trees in general).
\end{remark}

Let $\Tree$ be a tree diagram and $\iota$ its involution. The set of
its external vertices is naturally decomposed as the disjoint union of
three sets
\begin{equation}
  \label{eq:externaldecomp}
V_e(\Tree) = B \cup R_0 \cup R_1,  
\end{equation}
where $B$ are the blue vertices, $R_0$ are the red vertices fixed by
the involution, and $R_1$ are the red vertices not fixed by the
involution (an even set).  Let $b =|B|$ (its cardinality), $r = |R_0|$
and $f = \frac{|R_1|}{2}$. The main goal of the next section is to
provide an algorithm to compute (equivalence classes of) tree diagrams
with parameters $(b,r,f)$.

%

\section{Bi-valent trees}
\label{sec:bi-valent}
\begin{defi}
  A \emph{bi-valent tree} is a tree satisfying that the valence
  function on vertices takes at most two different values. 
\end{defi}

Clearly the valence function on any tree with at least three vertices
must take at least two different values, hence the valence function on
a bi-valent tree with at least three vertices takes precisely two
different values.
\begin{defi}
Let $T$ be any bi-valent tree.  The set of
\emph{external vertices} (resp. \emph{internal vertices}) that we
denote by $V_e$ (resp. $V_i$) is the set of vertices of $T$ having
valence $1$ (resp. having valence $>1$).
\end{defi}
Denote by $\Bi(m,n)$ the set of bi-valent trees
(up to isomorphism) with valences set $\{1,n\}$ made of $m$ internal vertices.
\begin{exam}
\label{ex:tree}
  There is a unique bi-valent tree $T$ in $\Bi(2,3)$ given by the graph shown in Figure~\ref{fig:ex1}.
   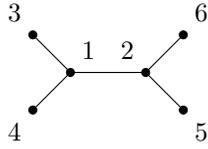
\begin{figure}[H]
   \centering
   \begin{tikzpicture}
     \draw (0,0) node[label={-120:$4$},inner sep=0pt,draw=black,circle,minimum size=3pt,fill=black]{} -- (0.5,0.5) node[label={60:$1$},inner sep=0pt,draw=black,circle,minimum size=3pt,fill=black]{} -- (1.5,0.5) node[label={120:$2$},inner sep=0pt,draw=black,circle,minimum size=3pt,fill=black]{} -- (2,0) node[label={-60:$5$},inner sep=0pt,draw=black,circle,minimum size=3pt,fill=black]{};
     \draw (0,1) node[label={120:$3$},inner sep=0pt,draw=black,circle,minimum size=3pt,fill=black]{} -- (0.5,0.5) -- (1.5,0.5) -- (2,1) node[label={60:$6$},inner sep=0pt,draw=black,circle,minimum size=3pt,fill=black]{};
   \end{tikzpicture}
   \caption{Unique tree in $\Bi(2,3)$.}
\label{fig:ex1}
 \end{figure}
Its adjacency matrix equals
\[\left(
\begin{array}{cc|cccc}
\textcolor{red}{0} & \textcolor{red}{1} & 1 & 1 & 0 & 0\\
\textcolor{red}{1} & \textcolor{red}{0} & 0 & 0 & 1 & 1\\ \hline
1 & 0 & 0 & 0 & 0 & 0\\ 
1 & 0 & 0 & 0 & 0 & 0\\ 
0 & 1 & 0 & 0 & 0 & 0\\
0 & 1 & 0 & 0 & 0 & 0
\end{array}\right),
\]
where the red block is precisely the adjacency matrix of the internal
vertices.
\end{exam}

%
%
\begin{lemma}
\label{lemma:rel-iner-external}
  Let $T$ be a bi-valent tree, satisfying that the valence of any vertex belongs to the set $\{1,n\}$ for some $n>1$. Then
  \[
|V_e| = (n-2)|V_i|+2.
    \]
\end{lemma}

\begin{proof}
  By induction on the size $|V_i|$. If $|V_i|=1$ then there must be
  $n$ external vertices (all joined to the unique internal one), hence
  the result. Suppose that $m = |V_i|>1$ and that the result
  holds for all bi-valent trees with at most $m-1$ internal
  vertices. Let $v,w$ be internal vertices joined by an edge (such a
  pair always exists because $m>1$). Cut the tree into two disjoint
  trees $T_1, T_2$, by removing the edge $(v,w)$ and add to each of
  the new trees an external edge joining $v$ (respectively $w$). Now
  each tree is again a bi-valent tree, say with $r$ and $s$ internal
  vertices (so $r+s = m$). The inductive hypothesis implies that
  \[
|V_e(T_1)| = |V_i(T_1)|(n-2) + 2 \qquad \text{ and } \qquad |V_e(T_2)| = |V_i(T_2)|(n-2) + 2.
\]
But $\lvert V_i(T)\rvert = \lvert V_i(T_1)\rvert + \lvert V_i(T_2)\rvert$ and
$|V_e(T)| = |V_e(T_1)| + |V_e(T_2)|-2$ (because we add a new external
vertex to both $T_1$ and $T_2$), hence the result.
\end{proof}

The lemma implies that if $T \in \Bi(m,n)$ then its number of vertices equals
\[
|V(T)| = (n-1)m + 2.
  \]

\subsection{Automorphisms of bi-valent trees}
Let $T$ be a bi-valent tree, and let $T_i = (V_i,E_i)$ be the sub-tree
of internal vertices, namely the tree whose set of vertices
$V_i = V_i(T)$ is the set of internal vertices of $T$, and whose set of
edges $E_i$ is the set of all edges of $T$ between internal
vertices. Given an automorphism $\sigma$ of $\Aut(T)$, it makes sense to
restrict $\sigma$ to the subgraph $T_i$, giving a sequence
\begin{equation}
  \label{eq:restriction}
\begin{tikzcd}   1\longrightarrow\operatorname{Aut}_{e}\longrightarrow\operatorname{Aut}(T)\overset{\operatorname{Res}}{\longrightarrow}\operatorname{Aut}(T_{i}),
\end{tikzcd}
\end{equation}
where $\text{Res}$ is the restriction map (a group morphism) and
$\Aut_e$ is (by definition) its kernel. It turns out that the
sequence~(\ref{eq:restriction}) is exact and furthermore it
splits. Before proving these facts we need a characterization of the
subgroup $\Aut_e$.

Recall that given two vertices of a tree, there exists a unique reduced path
joining them, so there is a natural definition of \emph{distance} between
vertices. For latter purposes, if the bi-valent tree $T$ has at least three
vertices, we denote by
\[
  \Phi \colon V_e \to V_i,
\]
the function which assigns
to an external vertex $v$ the unique internal vertex which is at
distance one from it.
Let $v\in V_e$ be an external vertex, and let $\Two(v)$ be the
set of external vertices at distance at most $2$ from $v$. Since
$v \in \Two(v)$, $\Two(v) \neq \emptyset$. Consider the following function
\[
m\colon V_e \to [1,\ldots,n], \qquad m(v) = |\Two(v)|.
\]
For $1 \le i \le n$, let
\[
V_e^{(i)} :=\{v \in V_e \mid m(v) = i\}.
\]
Then we can decompose the set of external vertices as the disjoint
union of the sets $V_e^{(i)}$.  For $i \in [1,\ldots,n]$, define
  \[
r(i) = \max \{|S|:S\subseteq V_e^{(i)}\text{ and } d(v,w)\neq2
\text{ for all }v,w\in S\}.
\]

\begin{lemma}\label{lemma:kernelautom}
  There is a non-canonical group isomorphism:
  \[
    \Aut_e \simeq \prod_{i=1}^{n} \SS_i^{r(i)},
  \]
  where $\SS_i$ denotes the symmetric group on $i$ elements.
\end{lemma}
\begin{proof}
  Let $\sigma \in \Aut(T)$ be an automorphism in the kernel of the
  restriction map and let $v \in V_e^{(i)}$ be an external vertex. Let
  $w = \Phi(v)$ the the internal vertex adjacent to $v$. Since
  $\sigma$ is in the kernel of the restriction map, $\sigma(w) = w$.
  Then $\sigma(v)$ is an external vertex which is also adjacent to
  $w$ (since any automorphism preserves distances). In particular,
  $\sigma(v) \in \Two(v) \subseteq V_e^{(i)}$ and
  $\sigma$ induces a permutation of the elements of $\Two(v)$ (a
  set with $i$ elements). By definition, there exists elements
  $w_1,\ldots,w_{r(i)}$ such that the set $V_e^{(i)}$ can be
  written as the disjoint union
  \[
V_e^{(i)} = \bigsqcup_{j=1}^{r(i)} \Two(w_j),
\]
so the action of $\sigma$ on $V_e^{(i)}$ can be represented as $r(i)$
permutations on $\SS_i$. It is clear that such map is bijective,
i.e. permutations on the (disjoint union of the) sets
$\Two(w_j)$ induce an automorphism of the tree $T$ which fixes
all internal vertices.
\end{proof}

Let $T$ and $T'$ be two bi-valent trees in $\Bi(m,n)$.  For the
purposes of the present article, an \emph{endomorphism} between $T$
and $T'$ is a bijective map between both the set of vertices of $T$ to
the set of vertices of $T'$ and between the set of edges of $T$ to the
set of edges of $T'$ (with the usual compatibility condition). We
denote by $\End(T,T')$ the set of all such maps.

For computational purposes, all trees considered will have its set of edges
labeled by positive integers. This provides a total order on the
graph's set of vertices.
\begin{defi}
  Let $T$ and $T'$ be two trees with a total order on their set of
  vertices. We say that a map $\sigma \in \End(T,T')$ is $2$-ordered
  if it satisfies that for all pair of external vertices
  $v, w \in V_e(T)$ at distance two, it holds that if $v < w$ then
  $\sigma(v) < \sigma(w)$.
\end{defi}

\begin{prop}
  \label{prop:extension}
  Let $T$ and $T'$ be bi-valent trees in $\Bi(m,n)$ with a total order
  on its set of vertices. Let $T_i$ (resp. $T_i'$) be the sub-tree of
  inner vertices of $T$ (resp. $T_i'$ be the sub-tree of inner
  vertices of $T'$). Let $\phi\colon T_i \to T_i'$ be a morphism of
  graphs. Then $\phi$ can be extended in a unique way to a $2$-ordered
  morphism $\psi$ of $\End(T,T')$. Furthermore, if $T_i = T_i'$ and
  $T = T'$, the natural map $\psi\colon\Aut(T_i) \to \Aut(T)$ is a
  group morphism.
\end{prop}

\begin{proof}
  To avoid confusion, we denote by $\val$ the valence function on
  either tree $T$ or $T'$ and by $\val_i$ the valence function on
  their sub-tree of internal vertices.  Let $v \in V_e(T)$ be any
  external vertex and let $w = \Phi(v)$ be the internal vertex
  adjacent to it. Since $T$ is bi-valent and $w \in T_i$, $\val(w)=n$
  and $\val_i(w)<n$ (because $v$ is not in the internal sub-tree). The
  map $\phi$ preserves valences (since it is a morphism from $T_i$ to
  $T_i'$) hence $\val_i(\phi(w))<n$ and so it must be joined to an
  external edge $v'$ of $T'$. The two sets $\Two(v)$ and $\Two(v')$
  have the same cardinality (equal to $n-\val_i(w)$). Clearly any
  extension of $\phi$ must send the elements of $\Two(v)$ to the ones
  in $\Two(v')$, and since both sets are ordered sets, there exists a
  unique bijection between them preserving the order, providing the
  required extension.

  When $T= T'$ and $T_i = T_i'$, a priori the map
  $\psi\colon\Aut(T_i) \to \Aut(T)$ is only a map of sets, but
  since the composition of two $2$-ordered maps is a $2$-ordered map,
  and from the uniqueness of the extension, the map $\psi$ is actually
  a group morphism.
\end{proof}

\begin{thm}
\label{thm:auto}
  The sequence
\begin{equation}
  \label{eq:automorphisms}
  \begin{tikzcd}
    1\longrightarrow\operatorname{Aut}_{e}\longrightarrow\operatorname{Aut}(T)\overset{\Res}{\longrightarrow}\operatorname{Aut}(T_{i})\longrightarrow1,
  \end{tikzcd}
\end{equation}
is exact. Furthermore, the map $\Res$ has a section
$\psi\colon\Aut(T_i) \to \Aut(T)$, hence
$\Aut(T) \simeq\Aut_e \rtimes\Aut(T_i)$.
\end{thm}
\begin{proof}
  The existence of the section $\psi$ follows from
  Proposition~\ref{prop:extension}, providing the surjectivity of the
  map $\Res$. The last statement is a well known fact of split short
  exact sequences of groups.
\end{proof}

Recall that two graphs $G_1$ and $G_2$ are isomorphic if and only if
their adjacency matrices are conjugated by a permutation matrix. The
advantage of working with bi-valent trees is that the sub-tree of inner
vertices determines its isomorphism class uniquely (so instead of
computing with matrices of size $(n-1)|V_i|+2$ times $(n-1)|V_i|+2$,
we can work with matrices of size $|V_i| \times |V_i|$, where
$\{1,n\}$ is the valuation set). In our applications (to construct
subgroups of $\PSL_2(\ZZ)$) $n$ will be $3$, so will roughly speaking
half the size of our matrices.

\begin{coro} 
  Two bi-valent trees $T$ and $T'$ are isomorphic if and only if
  their internal vertices sub-trees $T_i$ and $T_i'$ are isomorphic.
\end{coro}
\begin{proof}
  Clearly, if $\phi\colon T \to T'$ is an isomorphism of graphs, its
  restriction to $T_i$ gives an isomorphism between $T_i$ and
  $T_i'$. Conversely, let $\phi$ be an isomorphism between $T_{i}$ and
  $T_{i}'$. By Proposition~\ref{prop:extension}, the map $\phi$ extends
  to a morphism between $T$ and $T'$, hence the statement.
\end{proof}

\subsection{Algorithm to compute bi-valent trees}
Let $m\ge 1$ and $n > 1$ be positive
integers. Lemma~\ref{lemma:rel-iner-external} implies that any element
of $\Bi(m,n)$ has $m(n-2)+2$ external vertices. Trees will be
represented by their adjacency matrix, with the convention that on a
tree with $m(n-2)+2$ vertices, the first $m$ labels are for the
internal vertices and the remaining ones for the external vertices.

%
\begin{defi}\label{defi:t-ext}
  Let $T\in \Bi(m,n)$ and let $v\in V_{e}(T)$. The \emph{extension of $T$
    at $v$}, that will be denoted by $\Ext(T,v)$, is the tree obtained
  by adding $n-1$ new vertices $v_1, \ldots, v_{n-1}$ to the set of
  vertices of $T$, and $n-1$ edges to the set of edges of $T$, joining
  $v$ with $v_i$ for each $i=1,\ldots,n-1$.
\end{defi}

It is clear from its definition that if $T \in \Bi(m,n)$ and
$v \in V_e(T)$ then $\Ext(T,v) \in \Bi(m+1,n)$.

%

\begin{lemma}
  \label{lemma:removal-vertex}
  Let $T \in \Bi(m,n)$ with $m>1$. Then there exists $v \in V_e$ such that $|\Two(v)| = n-1$.
\end{lemma}

\begin{proof}
  Recall the definition of the map $\Phi \colon V_e \to V_i$, which assigns
  to an external vertex $v$ its adjacent internal vertex. If $v \in V_e$, then $|\Two(v)| = |\Phi^{-1}(\Phi(v))|$. Clearly,
  \[
|V_e| = \sum_{v \in V_i}|\Phi^{-1}(v)|.
    \]
    If all sets $\Two(v)$ have cardinality smaller than $n-1$, then
    $|\Phi^{-1}(v)| \le n-2$, so
    \[
      |V_e| \le (n-2)|V_i| < (n-2)|V_i|+2 = |V_e|,
    \]
    a contradiction.
\end{proof}

\begin{thm}\label{algbif}
  The following algorithm gives a set of representatives of
  isomorphism classes of elements in $\Bi(m,n)$ for $m\ge 1$, $n>1$: 
   \begin{algorithmic}[1]
     \IF {$m = 1$} 
     \RETURN $\left(\begin{smallmatrix}0 & 1 & \cdots & 1 \\ 1 & 0 & \cdots & 0\\ \vphantom{\int\limits^x}\smash{\vdots} & \vphantom{\int\limits^x}\smash{\vdots} & \vphantom{\int\limits^x}\smash{\ddots} & \vphantom{\int\limits^x}\smash{\vdots} \\ 1 & 0 & \cdots & 0\end{smallmatrix}\right)$ $($an $(n+1)\times (n+1)$ matrix$)$
     \ENDIF
     \STATE $S \leftarrow \emptyset$
     \FOR{$T \in \Bi(m-1,n)$}
     \FOR{$v\in V_{e}(T)$}
     \STATE $S\gets S \cup \operatorname{Ext}(T,v)$
     \ENDFOR
     \ENDFOR
     \STATE Remove isomorphic elements from $S$.
     \RETURN $S$
   \end{algorithmic}
\end{thm}
\begin{proof}
  If $m=1$, the tree has a unique internal vertex, hence $n$
  external ones (joined to the internal vertex). In particular, there
  is only one possible tree whose adjacency matrix is the given
  one.

  Assume now that $m>1$. If $T \in \Bi(m,n)$, by
  Lemma~\ref{lemma:removal-vertex}, there exists a vertex
  $v \in V_e(T)$ such that $\Two(v)$ has exactly $n-1$ elements. Let
  $w = \Phi(v)$ be the internal vertex adjacent to $v$. Let $T'$ be
  the tree obtained from $T$ by removing all vertices of $\Two(v)$, 
  as well as the edges joining $w$ with elements of
  $\Two(v)$. Then $T' \in \Bi(m-1,n)$ and clearly $T$ is isomorphic to
  $\Ext(T',w)$. In particular, $T$ is obtained from extending an
  element of $\Bi(m-1,n)$ by an external vertex, so $T$ is isomorphic
  to an element of $S$.
\end{proof}

Actually, the previous algorithm can be improved since clearly if $v, w$ are
two external vertices of a bi-valent tree $T$ at distance $2$ (or
equivalently $w \in \Two(v)$), then the extension of $T$ by $v$ is
isomorphic to the extension of $T$ by $w$. More generally:

\begin{prop}\label{prop:mejora}
  Let $T \in \Bi(m,n)$, let $v,w \in V_e$ and let $\varphi\in\Aut(T)$
  be such that $\varphi(v)=w$. Then $\varphi$ induces an
  isomorphism
  $\tilde{\varphi}\colon\Ext(T,v)\to\Ext(T,w)$.
\end{prop}
\begin{proof}
  Let $H=\Ext(T,v)$ and $G = \Ext(T,w)$. Recall that $H$ is obtained
  by adding $n-1$ vertices to $T$, say $v_1,\ldots,v_{n-1}$, and $G$ is
  obtained by adding $n-1$ vertices to $T$, say $w_1,\ldots,w_{n-1}$. Let
  $\tilde{\varphi}\colon H \to G$ be the map given by
  \[
    \tilde{\varphi}(v)=\begin{cases}\varphi(v), & \text{if }v\in T,\\
      w_{i}, & \text{if }v=v_{i}.\end{cases}
\]
  It can easily be verified that $\tilde{\varphi}$ is a well-defined isomorphism between the two trees.
\end{proof}

\begin{remark}
In the particular case when $w \in \Two(v)$, the previous map
corresponds precisely to the \emph{flip} isomorphism in $\Aut_e$ (the
kernel of the restriction map) sending $v \leftrightarrow w$ and
fixing all other elements.
\end{remark}

Define an equivalence relation on $V_e$ by determining that $v \sim w$
if there exists $\sigma \in \Aut_e$ such that $\sigma(v) = w$. The
notation $[V_e/{\sim}]$ will be used to denote any set of
representatives for the equivalent classes. Then we have the following improvement of Theorem~\ref{algbif}.

\begin{thm}
  \label{thm:representatives-trees}
  The following algorithm gives the elements in $\Bi(m,n)$ for
  $m\ge 1$, $n>1$, up to isomorphism:
   \begin{algorithmic}[1]
     \IF {$m = 1$} 
     \RETURN $\left(\begin{smallmatrix}0 & 1 & \cdots & 1 \\ 1 & 0 & \cdots & 0\\ \vphantom{\int\limits^x}\smash{\vdots} & \vphantom{\int\limits^x}\smash{\vdots} & \vphantom{\int\limits^x}\smash{\ddots} & \vphantom{\int\limits^x}\smash{\vdots} \\ 1 & 0 & \cdots & 0\end{smallmatrix}\right)$ (an $(n+1)\times (n+1)$ matrix)
     \ENDIF
     \STATE $S \leftarrow \emptyset$
     \FOR{$T \in \Bi(m-1,n)$}
     \FOR{$v\in [V_{e}(T)/{\sim}]$}
     \STATE $S\gets S \cup \operatorname{Ext}(T,v)$
     \ENDFOR
     \ENDFOR
     \STATE Remove isomorphic elements from $S$.
     \RETURN $S$
   \end{algorithmic}
   The last step of the algorithm is done by checking whether
   two elements of $S$ have isomorphic sub-trees of inner vertices or not.
\end{thm}

%
%
%
%
%
\subsection{Orientations on bi-valent trees}
Recall that the set of vertices on a bi-valent tree is the union of
the external vertices $V_e$ (those with valence one) with the internal
ones $V_i$ (those having valence greater than one). 

\begin{defi}
  An \emph{orientation} on a bi-valent tree $T \in \Bi(m,n)$ is an
  orientation on its set of internal vertices, i.e. give for each
  vertex $v \in V_i$ an ordering of the set (of size $n$) consisting of
  vertices adjacent to $v$.
\end{defi}

The way we represent an orientation algorithmically is by adding to
each internal vertex $v \in V_i$ an $n$-cycle permutation $\sigma_v$.
Thus a bi-valent tree with an orientation is a pair consisting of the adjacency
matrix together with a list of $n$-cycles indexed by the internal vertices.

\begin{exam}
  Let $T$ the unique tree $($up to isomorphisms$)$ in $\Bi(2,3)$ $($from
  Example~\ref{ex:tree}$)$. Since $T$ is already presented as a planar
  graph, an orientation is given by the anti-clockwise choice at each
  vertex $($as mentioned in Remark~\ref{remark:orientation}$)$, namely
    \begin{figure}[H]
   \centering
   \begin{tikzpicture}
     \draw (0,0) node[label={-120:$4$},inner sep=0pt,draw=black,circle,minimum size=3pt,fill=black]{} -- (0.5,0.5) node[label={60:$1$},inner sep=0pt,draw=black,circle,minimum size=3pt,fill=black]{} -- (1.5,0.5) node[label={120:$2$},inner sep=0pt,draw=black,circle,minimum size=3pt,fill=black]{} -- (2,0) node[label={-60:$5$},inner sep=0pt,draw=black,circle,minimum size=3pt,fill=black]{};
     \draw (0,1) node[label={120:$3$},inner sep=0pt,draw=black,circle,minimum size=3pt,fill=black]{} -- (0.5,0.5) -- (1.5,0.5) -- (2,1) node[label={60:$6$},inner sep=0pt,draw=black,circle,minimum size=3pt,fill=black]{};
     \draw[dash pattern=on 3pt off 4pt,draw=red,decoration={markings, mark=at position 0.8 with {\arrow{latex}[length=2mm]}},postaction={decorate}] (0.5,0.5) circle (0.3);
     \draw[dash pattern=on 3pt off 4pt,draw=red,decoration={markings, mark=at position 0.7 with {\arrow{latex[length=2mm]}}},postaction={decorate}] (1.5,0.5) circle (0.3);
   \end{tikzpicture}
   \caption{Unique oriented tree in $\Bi(2,3)$ except isomorphism.}
   \label{OTree4}
\end{figure}
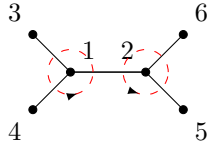
We represent this orientation of $T$ by the $3$-cycles
$\sigma_1=(2,3,4)$ and $\sigma_2=(1,5,6)$.
\end{exam}

\begin{defi}
  Let $T \in \Bi(m,n)$ be a bi-valent tree with an orientation and let $p=\{v_1,\ldots,v_n\}$ be a
  path between two external vertices. We say that the path
  $p$ is \emph{well oriented} if for each internal
  vertex $v_i$ in the path $p$ it holds that $\sigma_{v_i}(v_{i-1})=v_{i+1}$ (i.e. the
  orientation on the vertex $v$ sends the vertex of the path prior to
  $v$ to the subsequent one).
\end{defi}

\begin{remark}
  If $n=3$ (i.e. all vertices have valence $1$ or $3$, which is the
  case we are really interested in), a path is well oriented on a
  vertex $v_i$ if and only if the $3$-cycle $\sigma_{v_i}$ equals
  $(v_{i-1},v_{i+1},w)$, where $w$ is the third vertex adjacent to
  $v_i$.
\end{remark}

\begin{lemma}
  Let $T \in \Bi(m,n)$ be a bi-valent tree with an orientation. Let $v \in V_e$
  be an external vertex. Then there exists a unique $w \in V_e$ and a
  unique reduced and well oriented path $p$ between $v$ and $w$.
\end{lemma}

\begin{proof}
  Let $v_1$ be the internal vertex adjacent to $v$, and let
  $v_2 = \sigma_{v_1}(v)$, the unique choice so that the path
  $\{v,v_1,v_2\}$ is well oriented. If $v_2$ is an external vertex,
  then the path $\{v,v_1,v_2\}$ is a well oriented path between
  external vertices. Otherwise, let $v_3 = \sigma_{v_2}(v_1)$ (once
  again this is the unique choice so that the path $\{v,v_1,v_2,v_3\}$
  is well oriented) and continue this process. Since the number of
  vertices is finite, either at some point we reach an external
  vertex, or we get a ``period'', i.e. there exist $i<j$ such that
  $v_i = v_j$. Since $T$ is a tree, the only way to get a path
  starting and ending at the same vertex is that there exists an index
  $t$, with $i < t < j$ such that $v_{t-1} = v_{t+1}$, i.e.
  $v_{t+1} = \sigma_{v_t}(v_{t-1}) = v_{t-1}$, which cannot happen
  since the orientation at the vertex $v_t$ is an $n$-cycle (so has no
  fixed points).
\end{proof}

\begin{defi}
  Let $T \in \Bi(m,n)$ be an oriented bi-valent tree and let
  $v \in V_e$ be an external vertex. A vertex $w \in V_e$ is \emph{to
    the right} of $v$ if there exists a reduced and well oriented path $p$ from
  $v$ to $w$. Analogously, a vertex $w \in V_e$ is \emph{to the left}
  of $v$ if there exists a reduced and well oriented path from $w$ to $v$.
\end{defi}

In particular, we can define the function
\begin{equation}
  \label{eq:r-function}
r\colon V_e \to V_e, \qquad r(v) = w \text{ (vertex to the right of }v).
\end{equation}
Respectively, we have a function $l\colon V_e \to V_e$ sending $v$ to
the vertex to the left of $v$. Recall our definition of the map
$\Phi\colon V_e \to V_i$ which assigns to an external vertex the
internal one adjacent to it.

%
%
%
\begin{alg}
  Let $T\in \Bi(m,n)$ a bi-valent tree with an orientation and $v$ be
  an external vertex. The following algorithm computes the value of
  the function $r$ at an external vertex $v$:
   \begin{algorithmic}[1]
   \STATE $a \leftarrow v$
   \STATE $b\leftarrow\Phi(v)$ 
   \STATE $c\leftarrow\sigma_{b}(a)$
   \WHILE{$c\in V_{i}$}
   \STATE $a\leftarrow b$
   \STATE $b\leftarrow c$
   \STATE $c\leftarrow\sigma_{b}(a)$

   \ENDWHILE
   \RETURN $c$
   \end{algorithmic}
\end{alg}

%

\begin{lemma}
  Let $T \in \Bi(m,n)$ be a bi-valent tree with an orientation, and
  let $v \in V_e$ be an external vertex. Then
  \[
    V_e = \{r^i(v) : 0 \le i < |V_e|\}.
    \]
\end{lemma}

\begin{proof}
  Let $R(v)=\{v=v_1,v_2,\ldots,v_N\}$ be the set of external vertices
  to the right of each other (so $r(v_N)=v$). Let $p_i$ be the reduced
  and well oriented path between $v_{i}$ and $v_{i+1}$ for
  $i=1,\ldots,N-1$ and $p_N$ the reduced and well oriented path
  between $v_N$ and $v_1$. Each path is made of edges, so let
  $\tilde{E}$ be the union of all the (directed) edges appearing in
  $p_i$ for any $1 \le i \le N$ and let $\tilde{V}$ be the set of
  vertices of elements of $\tilde{E}$. Suppose we prove that for
  each $v \in \tilde{V}$ which is an internal vertex of $V$ it holds
  that its $n$ adjacent elements are also in $\tilde{V}$, then it
  must be the case that $\tilde{V} = V$. The reason is that given any
  $w \in V$, there is a unique reduced path joining $v$ to $w$. By our
  assumption, all edges of the path are elements of $\tilde{E}$, so
  $w \in \tilde{V}$ as claimed.

  To prove the stated property, note that since $r(v_N) = v$, the
  compositum $p_N \circ \cdots \circ p_1$ is a path between $v_1$ and
  $v_1$, hence is the trivial path, i.e. in this walk there is a
  complete cancellation of edges. In particular, if a directed edge
  lies in $\tilde{E}$, the edge with the opposite direction must also be a member
  of $\tilde{E}$.

  Let $b \in \tilde{V}$ be an internal vertex which is part of a path
  $p_i$, say $\{a,b,c\}$ is part of $p_i$. Since paths are well
  oriented, $\sigma_b(a)=c$. The complete cancellation property
  implies that the (directed) edge $(b,a)$ is part of some path $p_j$
  (for $1 \le j \le N$), i.e. $\{\sigma_b^{-1}(a),b,a\}$ appears in
  the path $p_j$ (because $p_j$ is also well oriented). This proves
  that if $(a,b) \in \tilde{E}$ then $(\sigma_b^{-1}(a),b)$ also
  belong to $\tilde{E}$. Repeating the argument,
  $(\sigma_b^{-i}(a),b) \in \tilde{E}$ for all $1 \le i \le n$, and
  since $\sigma_b$ is an $n$-cycle, all these edges are different
  hence all the vertices adjacent to $b$ are in $\tilde{V}$.
\end{proof}
%
%
\begin{remark}
  As mentioned before, we are mainly interested in the case $n=3$. The
  way we compute representatives for elements in $\Bi(m,3)$ with an
  orientation is the following: start with a set of isomorphism
  classes representatives $S$ for $\Bi(m,3)$ as described in
  Theorem~\ref{thm:representatives-trees}. Given an element
  $T \in S$, compute for each internal vertex the two possible
  orientations, and store the resulting bi-valent trees with
  an orientation. This produces a set whose number of elements is
  $2^{|V_i|}$ times the size of $S$.

  One can do better by the following observation: if $v$ is an
  internal vertex adjacent to two external vertices $w_1,w_2$, then
  the two possible orientations at the vertex $v$ are isomorphic (via
  the isomorphism sending $w_1 \leftrightarrow w_2$ and fixing the
  other vertices). For this reason, our code just picks one
  orientation for each internal vertex adjacent to two exterior ones. This
  trivial improvement gives in practice a huge saving.
\end{remark}

\subsection{Coloring external vertices} Recall from
Definition~\ref{defi:tree-diagram} that a tree diagram is an element
$\Tree$ of $\Bi(m,3)$ (for some $m$) with an orientation and a particular
coloring on the set $V_e$ (the external vertices of $\Tree$). Let $b, r, f$
be non-negative integers such that $b+r+2f = m+2 = |V_e|$. A coloring
on $V_e$ with parameters $(b,r,f)$ is equivalent to a decomposition of the set
$V_e$ as a disjoint union of the form
\[
  V_{e}=B\cup R_0\cup_i F_i,
\]
(as in (\ref{eq:externaldecomp})) where
\begin{enumerate}
\item the set $B$ has $b$ vertices; its vertices are the blue ones,
  
\item the set $R_0$ has $r$ vertices; its vertices are the red ones
  fixed by the involution $\iota$, and
  
\item each set $F_i$ has two red vertices, and the involution $\iota$
  sends one to the other. The union of the sets $F_i$ equals $R_1$ in
  (\ref{eq:externaldecomp}). It is a set of size $2f$ corresponding to
  the ``free sides''.
\end{enumerate}

\vspace{2pt}

Given $T \in \Bi(m,3)$ with an orientation, the way to compute all
possible coloring with parameters $(b,r,f)$ is first to compute all
possible subsets $B$ of $V_e$ of size $b$, and for each choice,
compute all possible choices of the subset $R_0$ of $r$ elements on
its complement. The complement $S=V_e \setminus (B \cup R_0)$ has then
$2f$-elements. A way of decomposing the set $S$ as a disjoint union of
$f$ subsets of size $2$ will be called a \emph{$2$-partition}.

\begin{problem}
Given a set $S$ with $2f$ elements, how to compute its $2$-partitions?  
\end{problem}


\begin{defi}
  Let $n$ be a positive integer. A $2n$-tuple $(v_1,\ldots,v_{2n})$ of
  integers is \emph{ordered} if it satisfies the following two properties:
  \begin{itemize}
  \item The sub-tuple $(v_1,v_3,\ldots,v_{2n-1})$ made of the odd
    entries is ordered.
  \item For any odd index $i$, $1 \le i \le 2n-1$, $v_i < v_{i+1}$.
  \end{itemize}
  Similarly, we say that an element of $\SS_{2f}$ is ordered if so is the
  $2f$-tuple it represents.
\end{defi}

Let $S=\{1,\ldots,2f\}$ and let $\Perm$ denote the set of ordered
$2f$-tuples made of elements of $S$. In particular, the coordinates of
elements of $S$ are all different. Let
$\Psi\colon \Perm \to \{2\text{-partitions}\}$ be the map given by
\begin{equation}
  \label{eq:partitions}
  \Psi(v_1,\ldots,v_{2f})=\{v_1,v_2\} \cup \cdots \cup \{v_{2f-1},v_{2f}\}.
\end{equation}

\begin{lemma}
  With the previous notation, the map $\Psi$ is bijective.
\end{lemma}
\begin{proof}
  Injectivity: let $(v_1,\ldots,v_{2f})$ and $(w_1,\ldots,w_{2f})$ be
  elements of $\Perm$ such that
  \[
    \Psi(v_1,\ldots,v_{2f}) = \Psi(w_1,\ldots,w_{2f}).
  \]
  Then
  $\{\{v_1,v_2\},\ldots,\{v_{2f-1},v_{2f}\}\} =
  \{\{w_1,w_2\},\dots,\{w_{2f-1},w_{2f}\}\}$, say
  $\{v_1,v_2\} = \{w_i,w_{i+1}\}$ for some odd index $i$. The second
  condition of an ordered tuple implies that $v_1<v_2$ and
  $w_i < w_{i+1}$ so $v_1 = w_i$ and $v_2 = w_{i+1}$. Repeating this
  argument it follows that the set of odd entries of the first
  tuple must equal the set of odd entries of the second one, and since
  both sets are ordered, the tuples must be the same.

  Surjectivity: let $S = \bigcup_{i=1}^f\{v_{2i-1},v_{2i}\}$ be a
  $2$-partition. Clearly we can order each set so that
  $v_{2i-1} < v_{2i}$ for all values of $i$, and we can also order the
  sets so that $v_{2i-1}< v_{2i+1}$ for all $i=1,\dots,f-1$. By definition, 
  the tuple $(v_1,\ldots,v_{2f})$ is ordered (so is an element of $P$).
\end{proof}

\begin{prop}
  Let $T\in \Bi(m,3)$ be a bi-valent tree with an orientation, and let
  $(b,r,f)$ be non-negative integers satisfying that $b+r+2f=m+2$. The
  following algorithm computes all possible coloring on $V_e$ with
  parameters $(b,r,f)$:
  
\begin{algorithmic}[1]
\STATE $C \leftarrow \emptyset$
\FOR {$B \subseteq V_e$, $|B| = b$}
\STATE $R=V_e \setminus B$
\FOR {$R_0 \subseteq R$, $|R_0| = r$}
\STATE $R_1 = R \setminus R_0$
\FOR {$\sigma\in \SS_{R_1}$}
\IF {\text{$\sigma$ is ordered}}
\STATE $C \leftarrow C\cup (B,R_0,\Psi(\sigma))$
\ENDIF
\ENDFOR
\ENDFOR
\ENDFOR
\RETURN $C$
\end{algorithmic}
\end{prop}

Let $\Tree$ be a tree diagram coming from a subgroup $\Sub$ of
$\PSL_2(\ZZ)$. Then 
\begin{equation}
  \label{eq:index-coloring}
[\PSL_2(\ZZ):\Sub]=3m+b,  
\end{equation}
where $m$ is the number of internal vertices of $\Tree$ and $b$ is the
number of blue external vertices (see for example the proposition in 
page 1078 of \cite{Kulkarni1991}). This implies that there are finitely many
triples $(b,r,f)$ corresponding to index $d$-subgroups of
$\PSL_2(\ZZ)$.

\begin{prop}
  Let $d$ be positive integer. The following algorithm computes all
  possible triples $(b,r,f)$ attached to index $d$ subgroups of
  $\PSL_2(\ZZ)$:
\begin{algorithmic}[1]
\STATE $M \leftarrow\emptyset$
\FOR{$\lceil\frac{d-2}{4}\rceil\leq m\leq\lfloor\frac{d}{3}\rfloor$}
\STATE $b\leftarrow d-3m$
\FOR {$j=0,\ldots,\lfloor \frac{m-b}{2} \rfloor +1$}
\STATE $M\leftarrow M\cup\{(b,m+2-b-2j,j)\}$
\ENDFOR
\ENDFOR
\RETURN $M$
\end{algorithmic}
  \end{prop}
\begin{proof}
  Let $\Sub$ be a subgroup of $\PSL_2(\ZZ)$ of index $d$, and let
  $m+2$ be the number of cusps belonging to a special polygon. The
  formula $d=3m+b$ together with the fact that the number of odd vertices cannot exceed
  the number of cusps, give the restriction
  \[
  3m\leq d\leq 4m+2.
  \]
  This completes the proof. 
\end{proof}

Combining the different algorithms described in the present section,
we can compute given a positive integer $d$ the set of non-equivalent
tree diagrams corresponding to subgroups of index $d$ in $\PSL_2(\ZZ)$.

\section{Tree diagrams, generalized Farey symbols and subgroups}
\label{section:treetosubgroup}

As explained in the introduction, our tree diagram is obtained as the
tree of $f$-edges of a special polygon. It is a natural question how to
recover information of the special polygon that is missing in the tree
diagram. For example, how can we compute it set of cusps? (note that
since the special polygon attached to a subgroup is not unique, the
answer will also not be unique).

As explained in \cite{Kulkarni1991}, the cusps of a special polygon
form what is called a \emph{generalized Farey symbol}, since if
$\frac{a}{b}$ and $\frac{c}{d}$ are two consecutive cusps of a special
polygon then $|ad-bc|=1$. 

\begin{defi}
  A \emph{generalized Farey symbol} (g.F.s. for short) is an expression
  of the form
  \[
    \{-\infty,c_2,c_3,\ldots,c_n,\infty\}
  \]
  where
  \begin{enumerate}
  \item $c_2$ and $c_n$ are integers, and some $c_i = 0$,
    
  \item $c_i = \frac{a_i}{b_i}$ are rational numbers in their reduced
    forms and ordered according to their magnitudes, such that
    \[
|a_ib_{i+1}-b_ia_{i+1}| = 1, \qquad i=2,3,\ldots,n-1.
      \]
  \end{enumerate}
\end{defi}

\begin{remark}
  Our definition is a little different from Kulkarni's one. The
  elements $-\infty$ and $\infty$ are identified as points on the
  polygon, so they correspond to the same point at infinity, but for
  algorithmic purposes it is better to represent them as different
  elements. In particular, we will represent $-\infty$ by the fraction
  $\frac{-1}{0}$ and $\infty$ by $\frac{1}{0}$. This is consistent
  with the \textsf{GAP} package developed in \cite{MR2735087} for computing with
  generalized Farey symbols.
  \label{rem:gfs-infty}
\end{remark}

Any path between two cusps (members of a g.F.s.)\!\! will necessarily go
through an end point of a tree diagram. As suggested in
\cite{Kulkarni1991}, one can add the information of the vertex color
to the g.F.s. forming what is called a \emph{Farey symbol}.
\begin{defi}
  A \emph{Farey symbol}
is an expression
of the form:
 \begin{figure}[h]
   \centering
   \begin{tikzpicture}
     \draw[thick,bend right] (0,0) node[label={$\{-\infty$},inner
     sep=0pt]{} to node[label={below:$p_{1}$},inner sep=0pt]{}
     (0.8,0) node[label={$c_2$},inner sep=0pt]{} to
     node[label={below:$p_2$},inner sep=0pt]{} (1.6,0)
     node[label={$c_3$},inner sep=0pt]{} to
     node[label={below:$p_3$},inner sep=0pt]{} (2.4,0)
     node[label={$c_{4}$},inner sep=0pt]{} to
     node[label={below:$p_4$},inner sep=0pt]{} (3.2,0)
     node[label={$\cdots$},inner sep=0pt]{} to
     node[label={below:$\cdots$},inner sep=0pt]{} (4,0)
     node[label={$c_n$},inner sep=0pt]{} to
     node[label={below:$p_{n}$},inner sep=0pt]{} (4.8,0)
     node[label={$\infty\}$},inner sep=0pt]{};
   \end{tikzpicture}
\end{figure}

\noindent where
\begin{itemize}
\item $\{-\infty,c_2,\ldots,c_n,\infty\}$ is a g.F.s.,
  
\item each symbol $p_i$ is one of $\circ$, $\bullet$ or a number label,
  depending on whether the line on the boundary that joins $c_{i}$
  with $c_{i+1}$ is even, odd or a free side (in which case the same
  label is used for its matching line).
\end{itemize}

\end{defi}


\begin{defi}
  Let $\Tree$ be a tree diagram, and let $v,w$ be two external vertices
  next to each other (i.e. one is to the left or right of the
  other). Define the \emph{bipartite distance} between $v$ and $w$
  (denoted by $\bd(v,w)$) as the distance between $v$ and $w$ in the
  tree obtained from $\Tree$ while constructing the bipartite cuboid graph
  before identifying the glued vertices.
\end{defi}

\begin{lemma}
  \label{lemma:bipartite-distance}
  Let $\Tree$ be a tree diagram, and let $v,w$ be external vertices next
  to each other. Let $d$ denote the distance between $v$ and $w$
  on the tree $\Tree$. Then
  \begin{equation}
    \label{eq:dist-relations}
\bd(v,w) = 2d -2 + \begin{cases}
     0, & \text{ if $v$ and $w$ are both even (red)},\\
     2, & \text{ if $v$ and $w$ are both odd (blue)},\\
     1, & \text{ otherwise}.
      \end{cases}
  \end{equation}
\end{lemma}

\begin{proof}
  Since $\Tree$ is a tree, there is a unique reduced path between $v$ and $w$, and
  since the distance between $v$ and $w$ is $d$, the path goes through
  $d-1$ internal vertices. The bipartite graph is obtained by adding
  an extra (red) vertex between each pair of internal vertices, giving
  an extra $d-2$ steps. At last, if either ending point is odd
  (i.e. blue), we need to add an extra vertex next to it.
\end{proof}

\begin{prop}
\label{prop:g.F.s.distance}
Let $\Tree$ be a tree diagram, and let $v,w$ be external vertices next to
each other. Let $x=\frac{a}{b}$, $y$ and $z=\frac{c}{d}$ be three
consecutive cusps of a special polygon attached to $\Tree$ satisfying that
$v$ is between $x$ and $y$ and that $w$ is between $y$ and $z$.  Then
\begin{equation}
  \label{eq:bip-dist}
    \bd(v,w)= 2|ad-bc| + \begin{cases}
     0, & \text{ if $v$ and $w$ are both even (red)},\\
     2, & \text{ if $v$ and $w$ are both odd (blue)},\\
     1, & \text{ otherwise}.
      \end{cases}
\end{equation}
\end{prop}

\begin{proof}
  Since everything is invariant under the action of $\SL_2(\ZZ)$, we
  can assume that $y$ is the infinity cusp. Then $x$ and
  $z$ are both integers (i.e. $b=d=1$) with $c<a$. The
  tessellation of the upper part of our special polygon then is one of
  the following:
  \begin{itemize}
  \item If $v$ and $w$ are both even (red) then it looks like
    Figure~\ref{fig:botheven}. Note that 
	\[
		2(a-c) = \bd(v,w)
	\]
	(since there are two
    $f$-edges between consecutive red vertical lines).
    
  \item If $v$ and $w$ are both odd (blue) then the tessellation looks like
    Figure~\ref{fig:bothodd}. Note that $ 2(a-c)+2 = \bd(v,w)$ (since there are two
    $f$-edges between consecutive blue vertical lines).
  \item If $v$ and $w$ have different colors (say $v$ is blue and $w$
    is red) then the tessellation looks like Figure~\ref{fig:hybridcase}. Note that
    $2(a-c)+1 = \bd(v,w)$.
  \end{itemize}
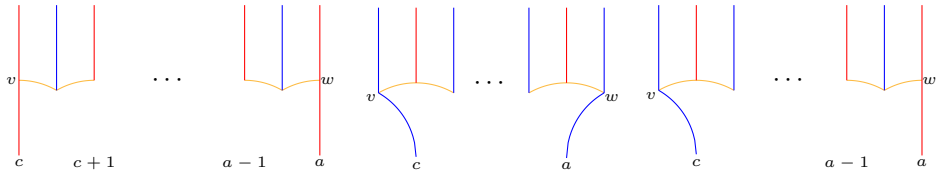
\begin{figure}[ht!]
  \begin{subfigure}{.3\textwidth}
   \centering
   \begin{tikzpicture}
   \draw[Dandelion] plot [domain=0:0.5] (\x,{sqrt(1-\x^2)}) plot [domain=0.5:1] (\x,{sqrt(1-(\x-1)^2)}) plot [domain=3:3.5] (\x,{sqrt(1-(\x-3)^2)}) plot [domain=3.5:4] (\x,{sqrt(1-(\x-4)^2)});
   \draw[blue] plot [domain=sqrt(3)/2:2] (0.5,\x) plot [domain=sqrt(3)/2:2] (3.5,\x); 
   \draw[red] (0,2) -- (0,0);
   \draw[red] (1,2) -- (1,1);
   \draw[red] (3,2) -- (3,1);
   \draw[red] (4,2) to (4,0);
  \draw (0,-0.1) node[font=\tiny]{$c$};
  \draw (1,-0.1) node[font=\tiny]{$c+1$};
  \draw (3,-0.1) node[font=\tiny]{$a-1$};
  \draw (4,-0.1) node[font=\tiny]{$a$};
  \draw (-0.1,1) node[font=\tiny]{$v$}; 
  \draw (4.1,1) node[font=\tiny]{$w$}; 
  \draw (2,1) node[]{$\cdots$};
  \end{tikzpicture}
  \caption{Case $v$ and $w$ even.}
  \label{fig:botheven}
\end{subfigure}%
\begin{subfigure}{.2\textwidth}
   \centering
   \begin{tikzpicture}
   \draw[Dandelion] plot [domain=0.5:1.5] (\x,{sqrt(1-(\x-1)^2)}) plot [domain=2.5:3.5] (\x,{sqrt(1-(\x-3)^2)});
   \draw[blue] plot [domain=sqrt(3)/2:2] (0.5,\x) plot [domain=sqrt(3)/2:2] (1.5,\x) plot [domain=sqrt(3)/2:2] (2.5,\x) plot [domain=sqrt(3)/2:2] (3.5,\x) plot [domain=0.5:1] (\x,{sqrt(1-\x^2)}) plot [domain=3:3.5] (\x,{sqrt(1-(\x-4)^2)}); 
   \draw[red] (1,2) -- (1,1);
   \draw[red] (3,2) -- (3,1);
  \draw (1,-0.1) node[font=\tiny]{$c$};
  \draw (3,-0.1) node[font=\tiny]{$a$};
  \draw (0.4,0.8) node[font=\tiny]{$v$}; 
  \draw (3.6,0.8) node[font=\tiny]{$w$}; 
  \draw (2,1) node[]{$\cdots$};
  \end{tikzpicture}
   \caption{Case $v$ and $w$ odd.}
  \label{fig:bothodd}
\end{subfigure}
\begin{subfigure}{.3\textwidth}
   \centering
   \begin{tikzpicture}
   \draw[Dandelion] plot [domain=0.5:1.5] (\x,{sqrt(1-(\x-1)^2)}) plot [domain=3:3.5] (\x,{sqrt(1-(\x-3)^2)}) plot [domain=3.5:4] (\x,{sqrt(1-(\x-4)^2)});
   \draw[blue] plot [domain=sqrt(3)/2:2] (0.5,\x) plot [domain=sqrt(3)/2:2] (1.5,\x) plot [domain=sqrt(3)/2:2] (3.5,\x) plot [domain=0.5:1] (\x,{sqrt(1-\x^2)}); 
   \draw[red] (1,2) -- (1,1);
   \draw[red] (3,2) -- (3,1);
   \draw[red] (4,2) to (4,0);
  \draw (1,-0.1) node[font=\tiny]{$c$};
  \draw (3,-0.1) node[font=\tiny]{$a-1$};
  \draw (4,-0.1) node[font=\tiny]{$a$};
  \draw (0.4,0.8) node[font=\tiny]{$v$}; 
  \draw (4.1,1) node[font=\tiny]{$w$}; 
  \draw (2.25,1) node[]{$\cdots$};
  \end{tikzpicture}
  \caption{Case $v$ odd and $w$ even.}
  \label{fig:hybridcase}
\end{subfigure}
\caption{Relation between the distance and the crossed product.}
\end{figure}
\end{proof}
From equations~(\ref{eq:dist-relations}) and (\ref{eq:bip-dist}) it
follows that if $x=\frac{a}{b}$, $y$ and $z=\frac{c}{d}$ are three
consecutive cusps with $v$ between $x$ and $y$ and $w$ between
$y$ and $z$, then
\begin{equation}
  \label{eq:gfs-rel}
  |ad-bc| = d(v,w)-1.
\end{equation}
These relations determine the g.F.s. uniquely (up to $\SL_2(\Z)$-equivalence).

\begin{thm}
\label{thm:tree-diagram2Farey-symbol}
Let $\Tree$ be a tree diagram, and $V_e=\{v_1,\dots,v_{m}\}$ denote the set
of external colored vertices. Assume that the set $V_e$ is ordered,
i.e. $v_{i+1}$ is to the right of $v_i$ for $1 \le i \le m$ $($with the
convention that $v_{m+1} = v_1)$. The following algorithm computes a
g.F.s. attached to $\Tree$:
  \begin{algorithmic}[1]
\STATE $F \leftarrow\{\frac{-1}{0},\frac{0}{1}\}$.
\IF {$m=2$}
\RETURN $F \cup \{\frac{1}{0}\}$.
\ELSE
\STATE $d \leftarrow d(v_1,v_2)$.
\STATE $F \leftarrow F \cup \{\frac{1}{d-1}\}$.
\FOR {$i=2,\ldots,m-2$}
\STATE $d\leftarrow d(v_{i},v_{i+1})$.
\STATE $\alpha \leftarrow \num(F[i+1])$.
\STATE $\beta\leftarrow \den(F[i+1])$.
\STATE $\gamma \leftarrow \num(F[i])$.
\STATE $\delta\leftarrow \den(F[i])$.
\STATE $\left(\begin{smallmatrix} y \\ x\end{smallmatrix}\right) \leftarrow \left(\begin{smallmatrix} \alpha & -\beta \\ \gamma & -\delta\end{smallmatrix}\right)^{-1}\left(\begin{smallmatrix} 1-d \\ -1\end{smallmatrix}\right)$.
\STATE $F \leftarrow F \cup \left\{\frac{x}{y}\right\}$.
\ENDFOR
\ENDIF
\RETURN $F \cup \left\{ \frac{1}{0}\right\}$.
\end{algorithmic}
\end{thm}

\begin{proof}
  Up to $\SL_2(\ZZ)$-equivalence, we can always assume that the first
  cusp (to the left of $v_1$) equals $\infty$ and the next one equals
  $0$. Since the number of cusps equals the number of external
  vertices, if there are only two such cusps, we are done. Suppose
  that there are at least three cusps, say
  $\left \{\frac{-1}{0},\frac{0}{1},\frac{a}{b}\right\}$, for some
  positive integers $a,b$. The integers $a,b$ satisfy the relations:
\[
  \begin{cases}
      |-b +0 \cdot a| &= d(v_1,v_2)-1,\\
        0\cdot b -a &= -1,    
\end{cases}
\]
where the second condition comes from~(\ref{eq:gfs-rel}). Then $a = 1$
and $b = d(v_1,v_2)-1$. If we constructed the first elements of our
g.F.s., say $\left\{\frac{-1}{0},\frac{0}{1},x_1,\ldots,x_i\right\}$,
with $x_{i-1} = \frac{\alpha}{\beta}$ and
$x_i = \frac{\gamma}{\delta}$, the next element $x_{i+1}=\frac{X}{Y}$
is a solution of the system
\[
  \begin{cases}
    |\alpha Y-\beta X| &= d(v_i,v_{i-1})-1,\\
    \gamma Y-\delta X &= -1.
\end{cases}
\]
The fact that $x_{i+1} > x_{i-1}$ together with the fact that both the
numerators and denominators of $x_{i-1}$ and $x_{i+1}$ are positive
integers (because $i \ge 2$) imply that $|\alpha Y-\beta X| = -\alpha Y+\beta X$.
\end{proof}

\begin{defi}
  A \emph{Kulkarni diagram} is an object consisting of a tree diagram,
  together with a choice of a g.F.s. (denoted by $G$) attached to it.
\end{defi}
%
The particular choice of the g.F.s. is not important, a different
choice will correspond to another subgroup of $\PSL_2(\Z)$ which is
conjugate (by a matrix in $\SL_2(\Z)$) to it. The elements of $G$ will
be called \emph{cusps}.
\subsection{Some important algorithms} The purpose of the present
section is to gather together some algorithms to compute with Kulkarni
diagrams as well as with its attached subgroup $\Sub$. The most
important algorithms are that of computing generators for $\Sub$, and
an algorithm to determine whether an element $g \in \PSL_2(\ZZ)$ belongs to
$\Sub$ or not (crucial to compute the passport attached to a Kulkarni diagram).

To make statements more clear, we will use the following notation:
$\Kul$ denotes a Kulkarni diagram,
$G=\{-\infty,c_2,\ldots,c_{m+2},\infty\}$ denotes its g.F.s. and
$\Sub$ the subgroup attached to it. As proven in \cite[Theorem
6.1]{Kulkarni1991}, the group $\Sub$ has a set of generators
$\{\alpha_1,\ldots,\alpha_{m+2}\}$ indexed by the elements of $G$,
representing how the paths in the boundary of a special polygon are
glued together.

\begin{alg}[Generator of a cusp] Let $\Kul$ be a kulkarni diagram,
  $V_e=\{v_1,\dots,v_{m+2}\}$ be its set of (external) colored and ordered vertices 
  and let $G$ be its underlying g.F.s. Given a cusp $c_k \in G$, the
  following algorithm computes the generator corresponding to the cusp
  $c_k$:
\begin{algorithmic}[1]
  \STATE $a \leftarrow\num(c_{k+1})$.
  \STATE $b \leftarrow\den(c_{k+1})$.
  \STATE $c \leftarrow\num(c_{k})$.
  \STATE $d \leftarrow\den(c_k)$.
\IF {\text{$v_k$ is ``even''}}
\STATE $g\leftarrow\left(\begin{smallmatrix}ab+cd & -c^2-a^2\\ b^2+d^2 & -ab-cd\end{smallmatrix}\right)$.
\ELSIF {\text{$v_k$ is ``odd''}}
\STATE $g\leftarrow\left(\begin{smallmatrix}ab+cb+cd & -c^2-ac-a^2\\ b^2+bd+d^2 & -ab-ad-cd\end{smallmatrix}\right)$;
\ELSE
\STATE {$v_k$ \text{ is identified with some $v_j$}}
  \STATE $a' \leftarrow \num(c_{j+1})$.
  \STATE $b' \leftarrow \den(c_{j+1})$.
  \STATE $c' \leftarrow \num(c_{j})$.
  \STATE $d' \leftarrow \den(c_j)$.
\STATE $g\leftarrow \left(\begin{smallmatrix}a'b+c'd &-c'c-a'a\\ d'd+b'b &-ab'-cd'\end{smallmatrix}\right)$.
\ENDIF
\RETURN $g$.
\end{algorithmic}
\end{alg}

\begin{proof}
  See \cite[Theorem 6.1]{Kulkarni1991}.
\end{proof}

\begin{alg}[Cusps representatives]
  \label{alg:cusp-orbits}
  Let $\Kul$ be a Kulkarni diagram and let $V_e=\{v_1,\dots,v_{m+2}\}$ be its set
  of (external) colored and ordered vertices. The following algorithm
  computes the subgroup of {$\mathbb{S}_{m+2}$} whose orbits correspond to
  equivalent cusps:
  \begin{algorithmic}[1]
  \STATE $S\leftarrow \emptyset$.
  \FOR {$v_k \in V_e$}
  \IF {$v_k \text{ is ``even" or ``odd"}$}
  \STATE $S \leftarrow S\cup \{(k,k+1)\}$.
  \ELSE
  \STATE {$v_k$ \text{is identified with some $v_j$}}
  \STATE $S \leftarrow S \cup \{(k,j+1),(k+1,j)\}$
  \ENDIF
  \ENDFOR
  \RETURN $\langle S \rangle$.
\end{algorithmic}
\end{alg}

\begin{proof}
  Clear from the gluing of the border paths.
\end{proof}
Let $C_1,\ldots,C_t$ be the cusp representatives of the Kulkarni diagram
$\Kul$, i.e. each set $C_i$ is made of elements of $G$ which are in the
$i$-th orbit under the action of the group $S$ computed using
Algorithm~\ref{alg:cusp-orbits}.

\begin{problem}
  How to compute $W(C_i)$, the width of $C_i$?
\end{problem}

As before, let $G=\{c_{1}=-\infty,c_2,c_3,\ldots,c_{m+2},c_{m+3}=\infty\}$. Let $d \colon G \to \ZZ$ be the function
\begin{equation}
  \label{eq:d-defi}
  d(c_i) = |a_{i-1}b_{i+1}-a_{i+1}b_{i-1}|,
\end{equation}
where $a_i$ (resp. $b_i$) is the numerator (resp. denominator) of the
cusp $c_i$, and the indices are taken ``in a cyclic order''
(i.e. $c_{m+3} = c_{1}$).

\begin{alg}[Width of a cusp]
  \label{alg:width}
  Let $\Kul$ be a Kulkarni diagram, $V_e=\{v_1,\ldots,v_{m+2}\}$ be
  its colored and ordered external vertices and let
  $c_{i_1},\ldots,c_{i_r}$ be the cusps belonging to the class
  $C$.  The following algorithm computes the width of $C$:
\begin{algorithmic}[1]
  \STATE $W \leftarrow 0$
  \FOR {$j=1,\ldots,r$}
  \STATE $e \leftarrow 0$
  \IF {$v_{i_j-1} \text{ is ``odd"}$}
  \STATE $e \leftarrow e+1$
  \ENDIF
  \IF {$v_{i_j} \text{ is ``odd"}$}
  \STATE $e \leftarrow e+1$
  \ENDIF
  \STATE $W \leftarrow W + d(c_{i_j}) + e/2$
  \ENDFOR
\RETURN $W$
\end{algorithmic}
\end{alg}

\begin{proof}
  See the proposition of page 1079 in \cite{Kulkarni1991}.
\end{proof}

\subsection{The word problem}
The classical solution to determine whether an element
$g \in \PSL_2(\ZZ)$ belongs to $\Sub$ (the subgroup attached to the
Kulkarni diagram $\Kul$) is to produce a ``reduction algorithm''
modulo elements of $\Sub$.  Based on the article \cite{Lang1995}, in
\cite{Kurth2007} (Algorithm, page 13) the authors give such an
algorithm. The problem with the stated method is that it has some
``border'' conditions which are not clearly stated, so we take the
opportunity to add to it the missing details. If $a$ and $b$ are
integers, we denote by $\quo(a,b)$ the cusp $\frac{a}{b}$, with the
usual convention that $-\infty = \frac{-1}{0}$ and
$\infty = \frac{1}{0}$.

\begin{thm}[Reduction]
  \label{thm:reduction}
  Let $\Kul$ be a Kulkarni diagram, let $\Sub$ be its attached group, $G$
  be its g.F.s. and let $C$ denote the vector consisting of the
  coloring on $G$. Let $g \in \PSL_2(\ZZ)$ be any element.  The
  following algorithm (Reduction) gives a \emph{reduced} element in
  the coset $\Sub g$:
  \begin{algorithmic}[1]
\STATE $N \leftarrow |G|$
\IF {$g[2,1] = 0$}
\STATE $case \leftarrow 1$
\ELSIF {$g[2,2] = 0$}
\STATE $case \leftarrow 2$
\ELSE
\STATE $case \leftarrow \text{``generic"}$
\ENDIF
\IF {$case = ``generic"$}
\IF {$g[1,2]/g[2,2] > g[1,1]/g[2,1]$}
\STATE {$x \leftarrow Reduction\left(\left(\begin{smallmatrix}-g[1,2]&g[1,1]\\-g[2,2] &g[2,1]\end{smallmatrix}\right)\right)$}
\RETURN $\left(\begin{smallmatrix}-x[1,2] &x[1,1]\\ -x[2,2]&x[2,1]\end{smallmatrix}\right)$
\ENDIF
\ELSIF {$case = 1$}
\IF{$g[1,1] > 0$}
\IF{$g[1,2]/g[2,2] < 0$}
\STATE{$x \leftarrow Reduction \left(\begin{smallmatrix} g[1,2] & -g[1,1]\\g[2,2] & -g[2,1]\end{smallmatrix}\right)$}
\RETURN $\left(\begin{smallmatrix}-x[1,2] &x[1,1]\\ -x[2,2]&x[2,1]\end{smallmatrix}\right)$
\ENDIF
\ELSE
\IF{$g[1,2]/g[2,2] > G[N-1]$}
\STATE {$g \leftarrow -g$}
\ELSE
\STATE {$x \leftarrow Reduction\left(\left(\begin{smallmatrix}-g[1,2]&g[1,1]\\-g[2,2] &g[2,1]\end{smallmatrix}\right)\right)$}
\RETURN $\left(\begin{smallmatrix}-x[1,2] &x[1,1]\\ -x[2,2]&x[2,1]\end{smallmatrix}\right)$
\ENDIF
\ENDIF
\ELSIF {$case = 2$}
\IF{$g[1,2] > 0$}
\IF{$g[1,1]/g[2,1] > 0$}
\STATE {$x \leftarrow Reduction\left(\left(\begin{smallmatrix}-g[1,2]&g[1,1]\\-g[2,2] &g[2,1]\end{smallmatrix}\right)\right)$}
\RETURN $\left(\begin{smallmatrix}-x[1,2] &x[1,1]\\ -x[2,2]&x[2,1]\end{smallmatrix}\right)$
\ELSE
\STATE {$g \leftarrow -g$}
\ENDIF
\ELSE
\IF{$g[1,1]/g[2,1] > G[N-1]$}
\STATE{$x \leftarrow Reduction \left(\begin{smallmatrix} g[1,2] & -g[1,1]\\g[2,2] & -g[2,1]\end{smallmatrix}\right)$}
\RETURN $\left(\begin{smallmatrix}-x[1,2] &x[1,1]\\ -x[2,2]&x[2,1]\end{smallmatrix}\right)$
\ENDIF
\ENDIF
\ENDIF
\STATE $a \leftarrow \quo(g[1,2], g[2,2])$
\STATE $b \leftarrow \quo(g[1,1], g[2,1])$
\IF {$a \in G \text{ and } b \in G$}
\RETURN $g$
\ELSE
\IF {$case = 1$}
\STATE $c \leftarrow N$
\ELSIF {$case = 2$}
\STATE $c \leftarrow 2$
\ELSIF {$b \in G$}
\STATE $c \leftarrow \text{Position}(G, b)$
\ELSE
\STATE $l \leftarrow G \cup \{b\}$ 
\STATE $Sort(l)$
\STATE $c \leftarrow Position(l, b)$
\ENDIF
\ENDIF
\IF {$not \; C[c-1] = ``odd"$}
\STATE $y \leftarrow CuspGenerators(G[c-1])$
\ELSE
\STATE $m \leftarrow (\num(G[c-1])+\num(G[c]))/(\den(G[c-1])+\den(G[c]))$
\IF{$g[2,1] = 0 \text{ or }b > m$}
\STATE $y \leftarrow CuspGenerator(G[c-1])$
\ELSE
\STATE $y \leftarrow CuspGenerators(G[c-1])^{-1}$
\ENDIF
\ENDIF
\RETURN $Reduction(y*g)$
\end{algorithmic}
\end{thm}
\begin{proof}
  Let
  $g = \left(\begin{smallmatrix} a & b\\ c &
      d\end{smallmatrix}\right)$, and let $\alpha = \quo(a,c)$,
  $\beta=\quo(b,d)$ be the cusps corresponding to the first and the
  second column. The boundary issues appear when either $\alpha$ or
  $\beta$ are the infinity cusp (depending also on whether they
  correspond to the $-\infty$ or the $\infty$ cusp). The first case
  corresponds to $\alpha = \infty$, the second case to
  $\beta = \infty$ and the remaining ones to the ``generic'' case. In
  \cite{Kurth2007} the authors start assuming that $\beta < \alpha$
  (following their main reference \cite{Lang1995}). If we multiply the
  matrix $g$ on the right by
  $S = \left(\begin{smallmatrix} 0 & 1\\-1 &
      0\end{smallmatrix}\right)$ has the effect of interchanging
  $\alpha \leftrightarrow \beta$, so in case $\alpha < \beta$ we
  reduce the matrix $gS$ and multiply the reduced matrix by $S$ to the
  right. Multiplication by $S$ sends
  \[
\begin{pmatrix} a & b\\ c & d \end{pmatrix} \to \begin{pmatrix} -b & a\\ -d & c\end{pmatrix}.
\]
This justifies the steps $10-12$ of the algorithm. In the non-generic
cases, this idea might not be enough to get the ``right'' reduction
step (as explained in \cite{Lang1995}, the idea behind the algorithm
is to shorten a distance). Suppose that $|\alpha|=\infty$ (so we are in the first case), then the border condition is the following:
\begin{itemize}
\item If $\alpha = \infty$ and $\beta$ is positive, we reduce $g$.
  
\item If $\alpha = \infty$ and $\beta$ is negative, we reduce $-gS$ (i.e. instead of considering the pair $(\infty,\beta)$ we work with $(\beta,-\infty)$).
  
\item If $\alpha = -\infty$ and $\beta$ is not larger than all elements of $G$, we reduce $gS$ (i.e. $(\beta,-\infty)$).
  
\item If $\alpha = -\infty$ and $\beta$ is larger than all elements of $G$, we reduce $-g$ (i.e. $(\infty,\beta)$).
\end{itemize}

Similarly, when $|\beta| = \infty$ (the second case), the border
condition is the following:
\begin{itemize}
\item If $\beta = -\infty$ and $\alpha$ is no larger than all elements of $G$, we reduce $g$.
  
\item If $\beta = -\infty$ and $\alpha$ is larger than all elements of $G$, we reduce $-gS$ (i.e. we replace the pair $(\alpha,-\infty)$ by $(\infty,\alpha)$).
  
\item If $\beta = \infty$ and $\alpha<0$, we reduce $-g$ (replacing $(\alpha,\infty)$ by $(\alpha,-\infty)$).
  
\item If $\beta = \infty$ and $\alpha >0$, we reduce $gS$ (replacing $(\alpha,\infty)$ by $(\infty,\alpha)$).
\end{itemize}
In all the above cases, we end up in a situation where
$\beta < \alpha$. The rest of the algorithm mimics the one presented
in \cite{Kurth2007}.
\end{proof}

Let $g \in \PSL_2(\ZZ)$ and let
$g'=\left(\begin{smallmatrix}a & b\\c & d\end{smallmatrix}\right)$ be
its reduction (i.e. the output of the last algorithm).
\begin{thm}
  \label{thm:isongroup}
  With the previous notation, an element $g \in \PSL_2(\ZZ)$ belongs
  to $\Sub$ if and only if one of the following is true:
\begin{enumerate}
\item $g' = \pm \left(\begin{smallmatrix} 1 & 0\\ 0 & 1\end{smallmatrix}\right)$,
  
\item $(\frac{b}{d},\frac{a}{c})$ is a free side paired with $(0,\infty)$,
  
\item
  $g'= \pm \left(\begin{smallmatrix} 0 & -1\\ 1 &
      0\end{smallmatrix}\right)$, and $0$ and $\infty$ are adjacent
  vertices with an even pairing between them.
\end{enumerate}
\end{thm}
\begin{proof}
  See \cite[Theorem 3.2]{Lang1995}.
\end{proof}
The last two results provide a solution to the word problem, namely
determining whether an element $\alpha \in \PSL_2(\ZZ)$ belongs to $\Sub$ or not. We
end this section with an algorithm to compute coset representatives
for $\Sub \backslash \PSL_2(\ZZ)$.

\begin{prop}[Coset representatives]
\label{prop:representatives}
  Let $\Kul$ be a Kulkarni diagram, let $\Sub$ be the attached group and
  $G$ be its g.F.s. Let $C$ be the vector made of the
  coloring of $G$.  The following algorithm gives a complete set of
  representatives for $\Sub \backslash \PSL_2(\ZZ)$:
  \begin{algorithmic}[1]
\STATE $N \leftarrow |G|$
\STATE $T \leftarrow \left(\begin{smallmatrix}1 & 1\\0 & 1\end{smallmatrix}\right)$
\STATE $reps \leftarrow \emptyset$
\FOR {$i =1,\ldots,N-1$}
\IF{$i = 1$}
\STATE $w \leftarrow |\num(G[N-1])*\den(G[2])-\den(G[N-1])*\num(G[2])|$
\ELSE
\STATE $w \leftarrow |\num(G[i-1])*\den(G[i+1])-\den(G[i-1])*\num(G[i+1])|$
\ENDIF
\STATE $A \leftarrow \left(\begin{smallmatrix}-\num(G[i])& \num(G[i+1])\\-\den(G[i])& \den(G[i+1])\end{smallmatrix}\right)$
\IF {$C[i] = ``odd"$}
\STATE $reps \leftarrow reps\cup \{A *T^{-1}\}$
\ENDIF
\FOR {$j \in [0,\ldots,w-1]$}
\STATE $reps \leftarrow reps \cup \{A* T^j\}$
\ENDFOR
\ENDFOR
\RETURN reps
\end{algorithmic}
\end{prop}

\begin{proof}
  The result is stated in $\S 5.3$ of \cite{Kurth2007}, although there
  is a minor mistake on it (which is why we present a proof). Let $\FF$ be the hyperbolic
  triangle with vertices $\rho$, $i$, $\infty$ as in
  Figure~\ref{tria} and let $P$ be the special polygon attached to $\Tree$
  with cusps the g.F.s. of $\Kul$.  If $c_i=\frac{a_i}{b_i}$ and
  $c_{i+1}=\frac{a_{i+1}}{b_{i+1}}$ are two consecutive vertices of
  the special polygon attached to $\Kul$ (so they are two consecutive
  cusps), then the matrix
  \[
    A=\begin{pmatrix} -a_i & a_{i+1}\\
      -b_i & b_{i+1}\end{pmatrix}
  \]
  sends the line joining $\infty$ and $0$ to the line joining $c_i$ to
  $c_{i+1}$ (note that the sign in the first column is missing in
  \cite{Kurth2007}). The image under $A^{-1}$ of $P$ then contains the
  cusp $A^{-1}\cdot c_{i-1}$ (a positive integer) next to $\infty$ next to the cusp $0$. If
  $c_{i-1} = \frac{a_{i-1}}{b_{i-1}}$, then
  \[
    A^{-1}\cdot c_{i-1} = w =|a_{i-1}b_{i+1}-b_{i-1}a_{i+1}|.
  \]
  Suppose that the line joining $c_i$ and $c_{i+1}$ is even,
  then the translates of $\FF$ appearing in $A^{-1}P$ are either of the form
\begin{figure}[t]
\begin{subfigure}{.3\textwidth}
   \centering
   \begin{tikzpicture}[scale=1]
   \draw[Dandelion] plot [domain=0:0.5] (\x,{sqrt(1-\x^2)}) plot [domain=0.5:1] (\x,{sqrt(1-(\x-1)^2)}) plot [domain=2:2.5] (\x,{sqrt(1-(\x-2)^2)}) plot [domain=2.5:3] (\x,{sqrt(1-(\x-3)^2)});
   \draw[blue] plot [domain=sqrt(3)/2:2] (0.5,\x) plot [domain=sqrt(3)/2:2] (2.5,\x);
   \draw[white] plot [domain=3:3.5] (\x,{sqrt(1-(\x-4)^2)});
   \draw[red] (0,2) -- (0,0);
   \draw[red] (1,2) -- (1,1);
   \draw[red] (2,2) -- (2,1);
   \draw[red] (3,2) to (3,0);
  \draw (0.25,1.5) node[font=\tiny]{$\mathcal{F}$}; 
  \draw (0.75,1.5) node[rotate=-90,font=\tiny]{$T(\mathcal{F})$};
  \draw (2.25,1.5) node[rotate=-90,font=\tiny]{$T^j(\mathcal{F})$};
  \draw (1.5,1.5) node[]{$\cdots$};
  \end{tikzpicture}
  \caption{Both even.}
  \label{fig:case1}
\end{subfigure}\hspace{-4em}
\begin{subfigure}{.3\textwidth}
   \centering
   \begin{tikzpicture}[scale=1]
   \draw[Dandelion] plot [domain=0:0.5] (\x,{sqrt(1-\x^2)}) plot [domain=0.5:1] (\x,{sqrt(1-(\x-1)^2)}) plot [domain=2:2.5] (\x,{sqrt(1-(\x-2)^2)}) plot [domain=2.5:3.5] (\x,{sqrt(1-(\x-3)^2)});
   \draw[blue] plot [domain=sqrt(3)/2:2] (0.5,\x) plot [domain=sqrt(3)/2:2] (2.5,\x) plot [domain=sqrt(3)/2:2] (3.5,\x) plot [domain=3:3.5] (\x,{sqrt(1-(\x-4)^2)}); 
   \draw[red] (0,2) -- (0,0);
   \draw[red] (1,2) -- (1,1);
   \draw[red] (2,2) -- (2,1);
   \draw[red] (3,2) to (3,1);
  \draw (0.25,1.5) node[font=\tiny]{$\mathcal{F}$};
  \draw (1.5,1.5) node[]{$\cdots$};
  \end{tikzpicture}
  \caption{Left even, right odd.}
  \label{fig:case2}
\end{subfigure}\hspace{-3em}
\begin{subfigure}{.3\textwidth}
   \centering
   \begin{tikzpicture}[scale=1]
   \draw[Dandelion] plot [domain=0.5:1] (\x,{sqrt(1-(\x-1)^2)}) plot [domain=2:2.5] (\x,{sqrt(1-(\x-2)^2)}) plot [domain=2.5:3] (\x,{sqrt(1-(\x-3)^2)});
   \draw[blue] plot [domain=sqrt(3)/2:2] (0.5,\x) plot [domain=sqrt(3)/2:2] (2.5,\x) plot [domain=0.5:1] (\x,{sqrt(1-\x^2)}); 
   \draw[red] (1,2) -- (1,1);
   \draw[red] (2,2) -- (2,1);
   \draw[red] (3,2) to (3,0);
  \draw (0.75,1.5) node[font=\tiny]{$\iota\mathcal{F}$};
  \draw (1.5,1.5) node[]{$\cdots$};
  \end{tikzpicture}
  \caption{Left odd, right even.}
\label{fig:case3}
\end{subfigure}\hspace{-3em}
\begin{subfigure}{.3\textwidth}
   \centering
   \begin{tikzpicture}[scale=1]
   \draw[Dandelion] plot [domain=0.5:1] (\x,{sqrt(1-(\x-1)^2)}) plot [domain=2:2.5] (\x,{sqrt(1-(\x-2)^2)}) plot [domain=2.5:3.5] (\x,{sqrt(1-(\x-3)^2)});
   \draw[blue] plot [domain=sqrt(3)/2:2] (0.5,\x) plot [domain=sqrt(3)/2:2] (2.5,\x) plot [domain=sqrt(3)/2:2] (3.5,\x) plot [domain=0.5:1] (\x,{sqrt(1-\x^2)}) plot [domain=3:3.5] (\x,{sqrt(1-(\x-4)^2)});
   \draw[red] (1,2) -- (1,1);
   \draw[red] (2,2) -- (2,1);
   \draw[red] (3,2) to (3,1); 
  \draw (0.75,1.5) node[font=\tiny,rotate=-90]{$\iota(\mathcal{F})$};
  \draw (3.25,1.5) node[rotate=-90,font=\tiny]{$T^w(\iota\mathcal{F})$};
  \draw (1.5,1.5) node[]{$\cdots$};
  \end{tikzpicture}
  \caption{Both odd.}
  \label{fig:case4}
\end{subfigure}
\caption{Representatives near the infinity cusp.}
\end{figure}
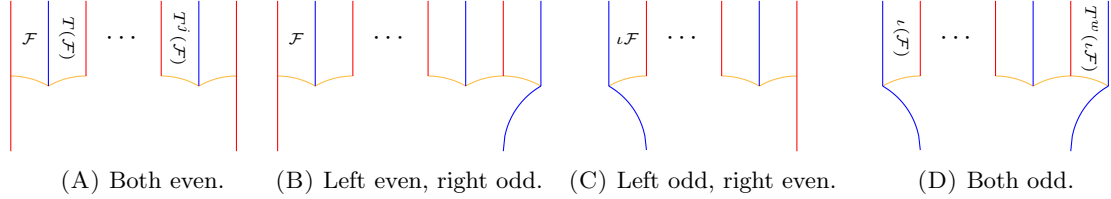  
\begin{itemize}
  \item $T^j(\FF \cup T\iota(\FF))$ if the line joining $c_{i}$ to
    $c_{i-1}$ is also even (so they are translates of a fundamental
    domain for $\PSL_2(\ZZ)$), where $0 \le j \le w-1$ (see Figure~\ref{fig:case1}),
    
  \item $T^j(\FF \cup T\iota(\FF))$, where $0 \le j \le w-1$ together
    with $T^w \FF$ (see Figure~\ref{fig:case2}) if the line joining $c_{i}$ to
    $c_{i-1}$ is odd. Note that the ``other'' part of the fundamental
    domain will appear while considering the cusp $c_{i-1}$, so we do
    not add it to the list to avoid repetitions.
  \end{itemize}
  In this case it is clear that the hyperbolic triangles in the
  special polygon $P$ containing the infinity point are 
  the translates by $A*T^j$ of $\FF$ (together with $T\iota(\FF)$) for
  $0 \le j \le w-1$.

  Suppose otherwise that the line joining $c_i$ and $c_{i+1}$ is odd. Then the translates
  of $\FF$ appearing in $A^{-1}P$ belong to one of the following two cases:

  \begin{itemize}
  \item $T^{-1}(T \iota(\FF))$ together with
    $T^j(\FF \cup T\iota(\FF))$, where $0 \le j \le w-1$ (see
    Figure~\ref{fig:case3}) if the line joining $c_i$ to $c_{i-1}$ is
    even. Note that we must add the first element as it was excluded
    from the opposite situation (namely left side even and right one
    odd).
  \item $T^{-1}(T \iota(\FF))$ together with $T^w\FF$ and
    $T^j(\FF \cup T\iota(\FF))$ for $0 \le j \le w-1$ (see
    Figure~\ref{fig:case4}). Once again, to avoid repetitions, we do
    not add the representative $T^w\FF$.
  \end{itemize}
  In this second case the hyperbolic triangles in the special polygon
  $P$ containing the infinity point are the translates of
  $A*T^j$ for $-1 \le j \le w-1$ of $\FF$ (together with
  $T\iota(\FF)$).
\end{proof}

\subsection{Relation with geometry}
From a finite index subgroup $\Sub$ of $\PSL_2(\ZZ)$ one can construct
the algebraic curve given by the quotient
$\calC_\Sub:=\Sub \backslash \Poincare^\ast$. Most of the geometric
invariants of $\calC_\Sub$ can be read from the Kulkarni diagram
$\Kul$. Define the following quantities:
\begin{itemize}
\item $e_2 = |\{\text{``red" vertices fixed by the involution in }R\}| = |R_0|$.
  
\item $e_3 = |\{\text{``blue" vertices in }V_e\}|$.
  
\item $t =$ the number of orbits of the group obtained as the output of
  Algorithm~\ref{alg:cusp-orbits}.
  
\item $d = [\PSL_2(\Z):\Sub]$.
\end{itemize}
Then in \cite[\S7]{Kulkarni1991} it is proven that the curve
$\calC_\Sub$ satisfies the following properties:
\begin{enumerate}
\item the number of branch points of $\Poincare \to \calC_\Sub$ of order $2$ equals $e_2$,
\item the number of branch points of $\Poincare \to \calC_\Sub$ of order $3$ equals $e_3$,
\item the number of inequivalent cusps of $\calC_\Sub$ equals $t$, and
\item the following formula holds:
\[
2g(\calC_\Sub) + t-1 = \frac{1}{2}|\{\text{``red" vertices non-fixed by the involution in }R\}|=f,
\]
where $g(\calC_\Sub)$ is the genus of $\calC_\Sub$.
\end{enumerate}
In particular, the colored set $V_e$ contains enough information to
compute $g(\calC_\Sub)$.

\section{Passports and equivalence classes of subgroups}
\label{sec:passports}

Let $G$ a group and $H$ be a subgroup of index $d$. Let
$X=\{g_1,\ldots,g_d\}$ be a set of left coset representatives for
$G /H$. Without loss of generality, we can assume that $g_1 \in H$. Then
\[
  G=\bigsqcup_{i=1}^dg_{i}H.
\]
There is a natural well known homomorphism of groups
\begin{align*}
	\theta_{H}\colon G&\longrightarrow \mathbb{S}_{X}\simeq \mathbb{S}_{d}\\
	g&\longmapsto\sigma_{g},
\end{align*}
where $\sigma_g$ is the automorphism of $X$ satisfying that
$g\cdot g_{i}H=\sigma_{g}(g_i)\cdot H$ for all $i$. The morphism
$\theta_H$ has important properties (see for example \cite[Theorems 5.3.1 and 5.3.2]{MR0414669}), namely:
\begin{itemize}
\item It determines the subgroup $H$ (with our choice  it is
  precisely the group of elements in $G$ fixing $g_1$).
  
\item Any conjugate of $H$ by an element of $G$ equals the group of
  elements in $G$ fixing $g_i$ for some $1 \le i \le d$.
  
\item If we denote by $H^{N}$ the biggest normal subgroup of $G$
  contained in $H$ (also known as the \emph{normal core of $H$ in $G$}),
  then it coincides with $\operatorname{Ker}(\theta_H)$.
  
\item If we denote by
  $\Sigma=\operatorname{Im}\theta_{H}\leq\mathbb{S}_{d}$ (that is
  isomorphic to $G/H^N$) then $\Sigma$ acts transitively on $\SS_d$.
\end{itemize}

Recall the following well known result on groups.

\begin{prop}
  Let $G$ be a group, and let $H, K$ be two index $d$ subgroups of
  $G$. $H$ is conjugated to $K$ by an element of $G$ if and only if there exists
  $\sigma\in\SS_{d}$ such that
  $\theta_{H}=\sigma\theta_{K}\sigma^{-1}$.
\end{prop}
\begin{proof} The proof is similar to that of \cite[Theorem
  5.3.3]{MR0414669}, although in such statement the author considers
  only faithful representations. Start supposing that there exists
  $g \in G$ such that $K=gHg^{-1}$. If $\{h_1,\ldots,h_d\}$ is
  a set of left coset representatives for $H$ then
  $\{gh_1g^{-1},\ldots,gh_dg^{-1}\}$ is a set of left coset
  representatives for $K$. Let $t \in G$ such that $t(h_i\cdot H) = h_j\cdot H$, then
  \[
    gtg^{-1} (gh_ig^{-1})\cdot K = (gh_jg^{-1}) \cdot K.
    \]
    Thus if $\sigma$ denotes the element of $\SS_d$ given by left
    multiplication by $g$, we get that
    $\sigma \theta_K \sigma^{-1} = \theta_H$ for the chosen set of
    representatives.

    Conversely, suppose that there exists $\sigma \in \SS_d$ such that
    $\theta_{H}=\sigma\theta_{K}\sigma^{-1}$. Without loss of
    generality (after reordering the left coset representatives of
    $K$), we can furthermore assume that $\theta_{H}=\theta_{K}$ and
    that the first coset representative for $H$ lies in $H$. In particular,
    if $h \in H$, $\theta_H(h)(1) = 1$. Then $\theta_{K}(h)(1)$, so if
    $k_1$ is the first left coset representative for $K$,
    $h (k_1\cdot K) = k_1 \cdot K$, so $h \in k_1Kk_1^{-1}$ and
    $H \subseteq k_1 K k_1^{-1}$. Since both groups have the same index
    in $G$, they must be equal.
\end{proof}

%

For a general group $G$ and a subgroup $H$, it is hard to describe the
map $\theta_H$, but if $G$ has a small (and known) number of
generators, then it is enough to compute the permutation corresponding
to each generator. This is indeed the case for $G = \PSL_2(\ZZ)$,
which is generated by the elements
\begin{equation}
  \label{eq:S-T-gens}
  S = \begin{pmatrix} 0 & -1 \\ 1 & 0\end{pmatrix}, \qquad T = \begin{pmatrix} 1 & 1\\ 0 & 1\end{pmatrix}.
\end{equation}
Then the pair $(\theta_H(S),\theta_H(T)) \in \SS_d\times \SS_d$ up to simultaneous
conjugation determines the $\SL_2(\ZZ)$-equivalence class of $H$.  It
is also useful to compute the permutation corresponding to the element
$R = ST$ (which is the composition of the two other permutations), as
it is an element of order $3$ in $\PSL_2(\ZZ)$ (with together with $S$
generates the group).

\begin{defi}
  A \emph{passport} is a tuple $(\sigma_S,\sigma_R,\sigma_T)$ in
  $\SS_d^3$ up to simultaneous conjugation satisfying that
  \begin{itemize}

  \item $\sigma_S^2=1$,
  
  \item $\sigma_R^3=1$,
  
  \item $\sigma_S \sigma_T = \sigma_R$,
  
  \item $\Sigma = \langle \sigma_S,\sigma_R\rangle$ is transitive.
  \end{itemize}
\end{defi}

An important result of Millington (\cite[Theorem 1]{Millington})
implies that the the passport contains a lot of arithmetic information
of the curve $H \backslash \calH$.

%
%
\begin{thm}
  There is a one-to-one correspondence between subgroups $H$ of index
  $d$ in $\PSL_2(\ZZ)$ and equivalence classes of passports
  $(\sigma_S,\sigma_R,\sigma_T)$. Furthermore, keeping the previous
  notation,
    \begin{enumerate}
    \item $e_2$ and $e_3$ are the number of elements fixed by the permutations $\sigma_S$
      and $\sigma_R$ respectively,
    \item $\sigma_T$ has $t$ disjoint cycles (where $t$ equals the
      number of cusps).
    \item Let $d_1,\dots,d_t$ be the lengths of the disjoint cycle
      decomposition of $\sigma_T$ (so $d=\sum_{i=1}^t d_i$). Then
      $\{d_1,\ldots,d_t\}$ equals the set (with repetitions)
      $\{w(C_1),\ldots,w(C_t)\}$ of cusp widths.
    \end{enumerate}
\end{thm}

Using the algorithms described in the previous section, it is easy to
give an algorithm to, given a Kulkarni diagram $\Kul$ together with an
element $\kappa \in \PSL_2(\ZZ)$, compute the permutation
$\sigma_\kappa$: use the algorithm of
Proposition~\ref{prop:representatives} to compute a set
$\{g_1,\ldots,g_d\}$ of left coset representative for
$\Sub \backslash \PSL_2(\ZZ)$. For each $1 \le i \le d$, use
Theorem~\ref{thm:isongroup} to determine the unique $j$ in
$\{1,\ldots,d\}$ such that $g_j^{-1}\kappa g_i$ belongs to
$\Sub$. Then $\sigma_\kappa(i)=j$. This algorithm is implemented in
our \textsf{GAP} package to compute the passport attached to a Kulkarni
diagram.

\subsection{Congruence subgroups}

Recall the following well known definition.
\begin{defi}
  A subgroup $\Sub$ of $\SL_2(\Z)$ is called a \emph{congruence subgroup} if there exists a positive integer $N$ such that the group
  \[
\Gamma(N)=\left\{\begin{pmatrix} a & b \\ c & d \end{pmatrix} \in \SL_2(\ZZ) \; : \; N\mid b, N\mid c, a \equiv 1 \Mod N, d\equiv 1 \Mod N \right\},
    \]
    is contained in $\Sub$. If such $N$ does not exist, we say that
    the subgroup is a non-congruence subgroup. The subgroup
    $\Gamma(N)$ is known as the \emph{main congruence subgroup of
      level $N$}.
  \end{defi}

  The main interest in congruence subgroups is that they have many
  endomorphisms (given by the Hecke operators). In particular, if
  $\Sub$ is a congruence subgroup, then the curve $\calC_\Sub$ is
  defined over a cyclotomic extension, and its Jacobian is isogenous
  to abelian varieties having very special properties (they are what
  is called of $\GL_2$-type, see for example \cite{MR1291394}).
  
\begin{defi}
  Let $\Kul$ be a Kulkarni diagram and $\{C_1,\ldots,C_t\}$ be the
  set of inequivalent cusps. The \emph{generalized level of $\Kul$} is
  defined to be the least common multiple of $\{W(C_1),\ldots,W(C_t)\}$.
\end{defi}

\begin{prop}
  If $\Sub$ is a congruence subgroup of generalized level $N$, then $\Gamma(N) \subseteq \Sub$.
\end{prop}
\begin{proof}
See \cite[Theorem 2]{Wohlfahrt1964}.
\end{proof}
%
%
For
completeness, we include the following algorithm due to Hsu for
determining whether a group $\Sub$ is a congruence group or not.
\begin{thm}
  Let $\Sub$ a subgroup of $\PSL_2(\ZZ)$ and let
  \[
    \sigma_L = \theta_\Sub\left(\left(\begin{smallmatrix}1 & 1\\ 0 & 1\end{smallmatrix}\right)\right),\qquad
    \sigma_R = \theta_\Sub\left(\left(\begin{smallmatrix}1 & 0\\ 1 & 1\end{smallmatrix}\right)\right).
  \]
    Let $N = e \cdot m$ be the order of $L$, where $e$ is a power of $2$ and
    $m$ is an odd positive integer.  Then the following algorithm
    determines whether $\Sub$ is a congruence subgroup or not:
\begin{algorithmic}[1]
\IF{$e=1$}
\IF{$\sigma_R^2\sigma_L^{-(N+1)/2}=1$}
\RETURN true
\ENDIF
\ELSIF{$m=1$}
\STATE $z\leftarrow\min\{x\in\mathbb{N}:N\mid1-5x\}$
\STATE $\tau\leftarrow\sigma_L^{20}\sigma_R^z\sigma_L^{-4}\sigma_R^{-1}$
\IF{$[\tau,\sigma_L\sigma_R^{-1}\sigma_L]=\tau^{-2}$ \AND $[\sigma_R,\tau]=\sigma_R^{24}$ \AND $(\tau\sigma_R^5\sigma_L\sigma_R^{-1}\sigma_L)^3=1$} 
\RETURN true
\ENDIF
\ELSE
\STATE $z\leftarrow\min\{x\in\mathbb{N}:e\mid1-5x\}$
\STATE $c\leftarrow\text{Solve}\begin{cases}x\equiv0\mod{e}\\x\equiv1\mod{m}\end{cases}$
\STATE $d\leftarrow\text{Solve}\begin{cases}x\equiv1\mod{e}\\x\equiv0\mod{m}\end{cases}$
\STATE $a\leftarrow\sigma_L^c$
\STATE $b\leftarrow\sigma_R^c$
\STATE $l\leftarrow\sigma_L^d$
\STATE $r\leftarrow\sigma_R^d$
\STATE $s\leftarrow l^{20}r^zl^{-4}r^{-1}$
\IF{$[a,r]=1$ \AND $(ab^{-1}a)^4=1$ \AND $(ab^{-1}a)^2=(b^{-1}a)^3=(b^2a^{-(m+1)/2})^3$
\AND $[s,lr^{-1}l]=s^{-2}$ \AND $[r,s]=r^{24}$ \AND $(lr^{-1}l)^2=(sr^5lr^{-1}l)^3$}
\RETURN true
\ENDIF
\ENDIF
\RETURN false
\end{algorithmic}
\end{thm}

\begin{proof}
  See \cite[Theorem 2.4]{HSUalgorithm} and $\S 3$ of loc. cit. for the implementation.
\end{proof}

\section{Some numerical data}
\label{sec:data}

We have systematically run our algorithm to compute Kulkarni
diagrams for subgroups up to index $20$. For each of them, we computed
the number of conjugacy classes for $\SL_2$ and $\GL_2$-subgroups. The
results are presented in Table~\ref{table:sizes}.
\begin{table}[]
\begin{tabular}{l|l|l|l}
  $n$ & $\#\Kul(n)$ & $SL_2(\Z)$-classes & $GL_2(\Z)$-classes \\
  \hline
  2   & 1        & 1                 & 1                 \\
3   & 2        & 2                 & 2                 \\
4   & 2        & 2                 & 2                 \\
5   & 1        & 1                 & 1                 \\
6   & 9        & 8                 & 8                 \\
7   & 8        & 6                 & 4                  \\
8   & 8        & 7                 & 6                  \\
9   & 54       & 14                & 12                  \\
10  & 101      & 27                & 19                  \\
11  & 80	   & 26                & 16                  \\
12  & 440      & 80                & 63                  \\
13  & 790      & 133               & 73                  \\
14  & 770      & 170               & 106                 \\
15  & 3184     & 348               & 213                  \\
16  & 6540     & 765               & 428                  \\
17  & 6582     & 1002              & 533                  \\
18  & 28958    & 2176              & 1277                 \\
19  & 61072    & 4682              & 2410                 \\
20  & 68920    & 6931              & 3679                \\
\end{tabular}
\caption{\label{table:sizes} Number of Kulkarni diagram vs $\SL_2$ and $\GL_2$ conjugacy classes.}
\end{table}
Note that the number
of Kulkarni diagrams grows much faster than the real number of
equivalent classes of subgroups. To our knowledge this is a phenomena
that was not observed before (since in the article published by
Kulkarni only subgroups with small index are computed). As mentioned
in the introduction, all these information can be downloaded from the GitHub
repository \url{https://github.com/vendramin/subgroups}.

In Table~\ref{table:subgroups} we give all the information obtained
from each $\SL_2(\Z)$-equivalence class for subgroups up to index $7$
(as it might be useful to the reader).  Some 
information phenomena obtained from our database:
\begin{itemize}
\item Among all subgroups (up to $\SL_2(\Z)$-equivalence), there are
  $90$ congruence ones (from a total of $16381$ subgroups),
  i.e. congruence groups are very rare even for small indices.
\item Up to index $20$, there are $2410$ groups corresponding to curves of genus $1$ (from a total of $16381$ subgroups).  
\item There are only $9$ non-conjugate subgroups whose curve has genus
  $2$. They all correspond to subgroups of index $18$ (as the example
  found in \cite{MR0337781}).   
\end{itemize}

Let us state some properties regarding the genus $2$ curves. Here is a list of their Kulkarni diagrams together with their passports:
\begin{enumerate}
\item The tree diagram equals \begin{minipage}{2.1cm}\includegraphics[width=60pt,height=30pt]{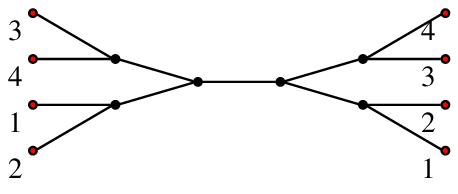}\end{minipage}, the
  g.F.s. equals
  $\{\infty, 0,1,\frac{4}{3},\frac{3}{2},\frac{5}{3},2,3,\infty\}$ and
  its passport equals
  \begin{itemize}
  \item $\sigma_S=(1,5)(2,8)(3,15)(4,9)(6,12)(7,16)(10,14)(11,17)(13,18)$,
    
  \item   $\sigma_R=(1,8,4)(2,15,7)(3,18,14)(5,12,9)(6,16,11)(10,17,13)$,
    
  \item   $\sigma_T=(1,2,3,13,14,15,16,17,18,10,11,12,4,5,6,7,8,9)$.
  \end{itemize}

\item The tree diagram equals \begin{minipage}{2.1cm}\includegraphics[width=60pt,height=30pt]{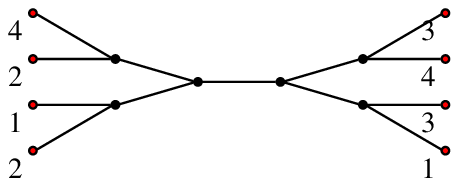}\end{minipage}, the
  g.F.s. equals
  $\{\infty, 0,1,\frac{4}{3},\frac{3}{2},\frac{5}{3},2,3,\infty\}$ and
  its passport equals
  \begin{itemize}
  \item $\sigma_S=(1,5)(2,8)(3,15)(4,18)(6,12)(7,16)(9,13)(10,14)(11,17)$,
    
  \item $\sigma_R=(1,8,4)(2,15,7)(3,18,14)(5,12,9)(6,16,11)(10,17,13)$,
    
  \item $\sigma_T=(1,2,3,4,5,6,7,8,18,10,11,12,13,14,15,16,17,9)$.
  \end{itemize}

\item The tree diagram equals \begin{minipage}{2.1cm}\includegraphics[width=60pt,height=30pt]{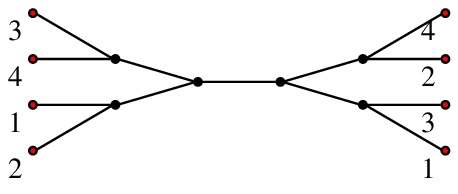}\end{minipage}, the
  g.F.s. equals
  $\{\infty, 0,1,\frac{4}{3},\frac{3}{2},\frac{5}{3},2,3,\infty\}$ and
  its passport equals
      \begin{itemize}
    \item $\sigma_S=(1,5)(2,8)(3,15)(4,10)(6,12)(7,16)(9,14)(11,17)(13,18)$,
      
    \item $\sigma_R=(1,8,4)(2,15,7)(3,18,14)(5,12,9)(6,16,11)(10,17,13)$,
      
    \item $\sigma_T=(1,2,3,13,4,5,6,7,8,10,11,12,14,15,16,17,18,9)$.
  \end{itemize}

\item The tree diagram equals \begin{minipage}{2.1cm}\includegraphics[width=60pt,height=30pt]{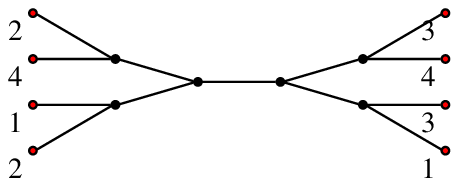}\end{minipage}, the
  g.F.s. equals
  $\{\infty, 0,1,\frac{4}{3},\frac{3}{2},\frac{5}{3},2,3,\infty\}$ and
  its passport equals
\begin{itemize}
  \item $\sigma_S=(1,5)(2,8)(3,15)(4,14)(6,12)(7,16)(9,13)(10,18)(11,17)$,
    
  \item  $\sigma_R=(1,8,4)(2,15,7)(3,18,14)(5,12,9)(6,16,11)(10,17,13)$,
    
  \item $\sigma_T=(1,2,3,10,11,12,13,18,4,5,6,7,8,14,15,16,17,9)$.
  \end{itemize}

\item The tree diagram equals \begin{minipage}{2.1cm}\includegraphics[width=60pt,height=30pt]{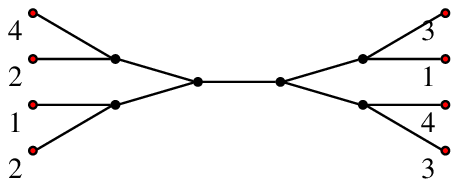}\end{minipage}, the
  g.F.s. equals
  $\{\infty, 0,1,\frac{4}{3},\frac{3}{2},\frac{5}{3},2,3,\infty\}$ and
  its passport equals
  \begin{itemize}
  \item $\sigma_S=(1,10)(2,8)(3,15)(4,18)(5,13)(6,12)(7,16)(9,14)(11,17)$,
    
  \item  $\sigma_R=(1,8,4)(2,15,7)(3,18,14)(5,12,9)(6,16,11)(10,17,13)$,
    
  \item $\sigma_T=(1,2,3,4,10,11,12,14,15,16,17,5,6,7,8,18,9,13)$.
  \end{itemize}

\item The tree diagram equals \begin{minipage}{2.1cm}\includegraphics[width=60pt,height=30pt]{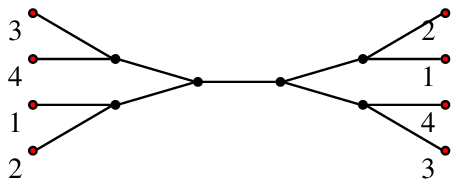}\end{minipage}, the
  g.F.s. equals
  $\{\infty, 0,1,\frac{4}{3},\frac{3}{2},\frac{5}{3},2,3,\infty\}$ and
  its passport equals
  \begin{itemize}
  \item $\sigma_S=(1,10)(2,8)(3,15)(4,13)(5,14)(6,12)(7,16)(9,18)(11,17)$,
  \item  $\sigma_R=(1,8,4)(2,15,7)(3,18,14)(5,12,9)(6,16,11)(10,17,13)$,    
  \item $\sigma_T=(1,2,3,9,14,15,16,17,4,10,11,12,18,5,6,7,8,13)$.
  \end{itemize}

\item The tree diagram equals \begin{minipage}{2.1cm}\includegraphics[width=60pt,height=30pt]{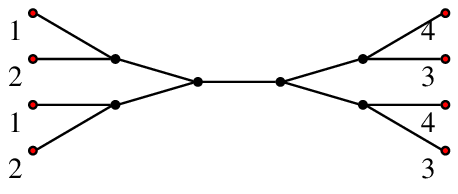}\end{minipage}, the
  g.F.s. equals
  $\{\infty, 0,1,\frac{4}{3},\frac{3}{2},\frac{5}{3},2,3,\infty\}$ and
  its passport equals
  \begin{itemize}
  \item $\sigma_S=(1,14)(2,8)(3,15)(4,18)(5,10)(6,12)(7,16)(9,13)(11,17)$,
  \item  $\sigma_R=(1,8,4)(2,15,7)(3,18,14)(5,12,9)(6,16,11)(10,17,13)$,
  \item $\sigma_T= (1,2,3,4,14,15,16,17,9,10,11,12,13,5,6,7,8,18)$.
  \end{itemize}

\item The tree diagram equals \begin{minipage}{2.1cm}\includegraphics[width=60pt,height=30pt]{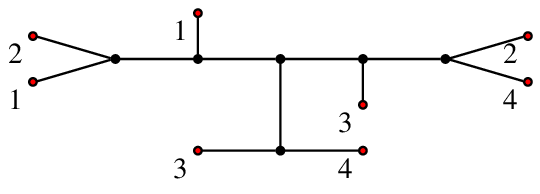}\end{minipage}, the
  g.F.s. equals
  $\{\infty, 0,\frac{1}{2},1,\frac{3}{2},2,3,4,\infty\}$ and
  its passport equals
  \begin{itemize}
  \item $\sigma_S=(1,7)(2,10)(3,14)(4,17)(5,18)(6,11)(8,13)(9,15)(12,16)$,
  \item  $\sigma_R=(1,10,6)(2,14,9)(3,17,13)(4,18,16)(5,11,7)(8,15,12)$,
  \item $\sigma_T=(1,2,3,4,5,6,7,18,12,13,14,15,16,17,8,9,10,11)$.
  \end{itemize}

\item The tree diagram equals \begin{minipage}{2.1cm}\includegraphics[width=60pt,height=30pt]{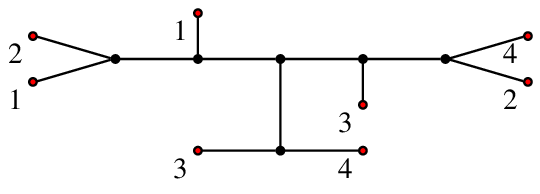}\end{minipage}, the
  g.F.s. equals
  $\{\infty, 0,\frac{1}{2},1,\frac{3}{2},2,3,4,\infty\}$ and
  its passport equals
  \begin{itemize}
  \item $\sigma_S=(1,7)(2,10)(3,14)(4,17)(5,16)(6,11)(8,13)(9,15)(12,18)$,
    
  \item  $\sigma_R=(1,10,6)(2,14,9)(3,17,13)(4,18,16)(5,11,7)(8,15,12)$,
  \item $\sigma_T=(1,2,3,4,12,13,14,15,18,5,6,7,16,17,8,9,10,11)$.
  \end{itemize}
\end{enumerate}

The order of the image of $\theta_\Sub$ in each case equals
$1008, 258048, 486, 4896, 258048, 648, 258048, 4896, 4896$
respectively. Furthermore, the fourth, the eighth one and the ninth
ones are isomorphic to $\PGL_2(\F_{17})$. Then this three groups must
be the ones found by Atkin and Swinnerton-Dyer in \cite[Table
1]{MR0337781} (corresponding to a curve defined over the cubic field
with generating polynomial $x^3-3x+1$).

Recall that the modular curve attached to a congruence subgroup has
many endomorphisms (corresponding to the so called Hecke operators),
while modular curves for non-congruence subgroups tend to not have
endomorphisms at all (there are no Hecke operators in the new part, as
proved in \cite{MR1302789}). In particular, for each of the previous
genus $2$ curves, if there is no contribution from smaller levels, one
expects the Jacobian to be absolutely simple when the subgroup is a
non-congruence one.

The third one and the seventh groups are congruence ones, hence they
do have many endomorphisms. Let us study what happens for the remaining
seven ones. Here is an elementary result to determine whether a group
is contained in a larger (proper) subgroup of $\PSL_2(\Z)$ in terms of
the passport representation.

\begin{lemma}
  \label{lemma:subgroup}
  Let $\Sub$ be a subgroup of $\PSL_2(\Z)$ of index $d$, with coset
  representatives $\{\mu_1,\ldots,\mu_d\}$. Let $\sigma_S$, $\sigma_T$
  be its permutation representation. Then there exists a subgroup
  $\tilde{\Sub}$ of $\PSL_2(\Z)$ of index $n$ (with $d \mid n$)
  containing $\Sub$ if and only if the set $\{\mu_1,\ldots,\mu_d\}$ can be
  written as a disjoint union of $n$ sets each of them with
  $\frac{d}{n}$ elements satisfying that $\sigma_S$ and $\sigma_T$
  preserve the partition.
\end{lemma}
\begin{proof}
  If there exists such a subgroup $\tilde{\Sub}$, let
  $\{g_1,\ldots,g_n\}$ be coset representatives for
  $\tilde{\Sub} \backslash \PSL_2(\Z)$ and
  $\{h_1,\ldots,h_{\frac{d}{n}}\}$ be coset representatives for
  $\Sub\backslash \tilde{\Sub}$, so that $\{h_j g_i\}$ are
  representatives for $\Sub \backslash \PSL_2(\Z)$. In particular, the
  set of representatives for $\Sub \backslash \PSL_2(\Z)$ can be
  written as the disjoint union
  \begin{equation}
    \label{eq:partition}
    \bigcup_{i=1}^n\{h_1g_i,\ldots,h_{\frac{d}{n}}g_i\}.    
  \end{equation}
  Let $\nu$ be any element of $\PSL_2(\Z)$. If
  $\tilde{\Sub}g_i\nu = \tilde{\Sub}g_j$ then the sets
  \[
    \{\Sub h_1g_i\nu,\ldots,\Sub h_{\frac{d}{n}}g_i\nu\}\text{ and }
  \{\Sub h_1g_j,\ldots,\Sub h_{\frac{d}{n}}g_j\}
  \] 
  must be the same,
  so $\sigma_S$ and $\sigma_T$ preserve the partition~(\ref{eq:partition}).

  Reciprocally, let $\{g_1,\ldots,g_d\}$ be coset representatives of
  $\Sub\backslash \PSL_2(\Z)$, and let
  $\theta_\Sub\colon\PSL_2(\Z) \to \SS_d$ be as before (identifying the set
  $\{g_1,\ldots,g_d\}$ with the set $\{1,\ldots,d\}$). Assume that
  $g_1 \in \Sub$, so that $\Sub$ corresponds precisely to the elements
  of $\PSL_2(\Z)$ whose image under $\theta_\Sub$ fix $1$.

  By hypothesis (after relabeling the indexes if necessary) the set
  $\{1,\ldots,d\}$ can be written as the disjoint union of the $n$
  sets
  $\{1,\ldots,\frac{d}{n}\}\cup \ldots \cup
  \{d-\frac{d}{n}+1,\ldots,d\}$, and our representation
  $\theta_{\Sub}$ induces a morphism $\widetilde{\theta_{\Sub}}$ of
  $\PSL_2(\Z)$ onto $\SS_n$ (via the action on the previous sets).  Let
  $\tilde{\Sub}$ the set of elements of $\PSL_2(\Z)$ that preserve the
  set $\{1,\ldots,\frac{d}{n}\}$. Clearly $\tilde{\Sub}$ contains
  $\Sub$, and the index of $\tilde{\Sub}$ on $\PSL_2(\Z)$ is $n$
  because the image of $\theta_{\Sub}$ is a transitive group (so the
  same holds for $\widetilde{\theta_{\Sub}}$).
\end{proof}

Let us compute for each of our genus $2$ subgroups $\Sub$ their
``old'' part. If there exists a subgroup $\tilde{\Sub}$ with index $9$
in $\PSL_2(\Z)$, then the lemma implies that our set of $18$ elements
can be split into $9$ sets of size $2$, which are preserved under the
action of $\PSL_2(\Z)$. Note that $\sigma_T$ is an $18$-cycle in all
cases. Suppose that such a partition exists, and that $\{a,b\}$ is a
pair of elements of it. Since $\sigma_T$ acts transitively, there
exists $i$ such that $b=\sigma_T^i(a)$. Applying $\sigma_T^i$ to the
pair $\{a,b\}$ we obtain the pair
$\{\sigma^i_T(a),\sigma^i_T(b)\} = \{b,\sigma_T^{2i}(a)\}$, so
$a = \sigma_T^{2i}(a)$, hence $i=9$. In particular, the partition must
be of the form $\{a,\sigma_T^9(a)\}$ for $a$ varying in the set
$\{1,\ldots,18\}$. Then we are led to verify whether the permutation
$\sigma_S$ preserves this partition or not. Consider each of the nine
cases separately.
\begin{enumerate}
\item We obtain the partition:
  \[
    \{1,10\}\{2,11\}\{3,12\}\{4,13\}\{5,14\}\{6,15\}\{7,16\}\{8,17\}\{9,18\}.
    \] Clearly
  the element $\sigma_S$ preserves the partition, so this contains the
  subgroup of index $9$ with passport $\sigma_S=(1,5)(2,8)(3,6)(4,9)$
  and $\sigma_R=(1,8,4)(2,6,7)(3,9,5)$. It matches the 12th element in
  the GitHub repository of subgroups of index $9$. It is easy to verify that it corresponds
  to a genus $1$ curve. In particular, the Jacobian of the curve is
  isogenous to the product of two elliptic curves.
  
\item A similar computation as the previous one proves that this
  second group is contained in an index $9$ subgroup with the same
  passport as the previous case. Once again, the surface is isogenous
  to the product of two elliptic curves.
  
\item In this case, the partition obtained from $\sigma_T$ is not
  compatible with the permutation $\sigma_S$, hence the group is not
  contained in a group of index $9$ (recall that this group is a
  congruence one).
  
\item The group is not contained in an index $9$ subgroup because once
  again the partition is not compatible with the action of $\sigma_S$.
  
\item The group is contained in a group of index $9$ with the same
  passport as the first two cases.
  
\item In this case, the group is contained in a subgroup of index $9$,
  with passport $\sigma_S=(2,8)(3,6)$,
  $\sigma_R=(1,8,4)(2,6,7)(3,9,5)$. Such a group corresponds to a
  curve of genus zero.
  
\item The group is contained in a group of index $9$ with the same
  passport as the first two cases.
  
\item The group is not contained in an index $9$ subgroup because the partition is not compatible with the action of $\sigma_S$.
\item The group is not contained in an index $9$ subgroup because the partition is not compatible with the action of $\sigma_S$.
\end{enumerate}

Although the unique subgroup of index $2$ has genus zero, a similar
computation shows that the groups $1, 3, 6$ are contained in the
subgroup of index $2$, while the other ones are not.

Verifying which groups are contained in subgroups of index $6$, we see
that the only ones are the first one (contained in the subgroup with
passport $\sigma_S=(1,5)(2,3)(4,6)$ and $\sigma_R=(1,3,4)$) and the sixth
one, contained in the subgroup with passport $\sigma_S=(1,4)(2,5)(3,6)$
and $\sigma_R=(1,5,3)(2,6,4)$, corresponding to the last subgroup of
index $6$ in the GitHub repository (see also
Table~\ref{table:subgroups}). Such a group corresponds to a curve of
genus one, hence the surface attached to it is isogenous to the
product of two elliptic curves as well. We deduce that the only
possible absolutely simple surface from the list of index $18$
subgroups are the ones corresponding to the groups $6$, $8$ and $9$
which match the ones found by Atkin and Swinnerton-Dyer (defined over
a cubic field), while all other ones are isogenous to the product of
two elliptic curves.

\begin{table}[]
\scalebox{0.6}{
  \begin{tabular}[t]{c|c|c|c|c|c|c|c|c|c}
Index & Kulkarni Graph & g.F.S. & Passport & Genus & $e_2$ & $e_3$ & Cusps \& Width & Generators & Congruence\\
  \hline
$2$ &   \includegraphics[width=50pt,height=10pt]{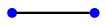} & $\{\infty, 0,\infty\}$ & $ (1,2), (), (1,2)$ & $0$ & $0$ & $2$ & $(\{\infty,0\}, 2)$ & 
 $\left(\begin{smallmatrix} -1 &  -1\\ 1& 0\end{smallmatrix}\right), \left(\begin{smallmatrix}0 &  -1 \\ 1 &  -1\end{smallmatrix}\right)$ & Yes\\
  \hline
 $3$ &   \raisebox{-.5\height}{\includegraphics[width=50pt,height=30pt]{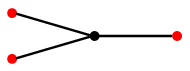}} & $\{\infty, 0,1,\infty\}$ & $ (), (1,3,2), (1,3,2)$ & $0$ & $3$ & $0$ & $(\{\infty,0,1\}, 3)$ & 
 $\left(\begin{smallmatrix} 0 &  -1\\ 1& 0\end{smallmatrix}\right), \left(\begin{smallmatrix}1 &  -1 \\ 2 &  -1\end{smallmatrix}\right), \left(\begin{smallmatrix}1 &  -2 \\ 1 &  -1\end{smallmatrix}\right)$ & Yes\\
    \hline
$3$ &    \raisebox{-.5\height}{\includegraphics[width=50pt,height=30pt]{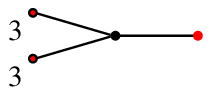}} & $\{\infty, 0,1,\infty\}$ & $ (2,3), (1,3,2), (1,2)$ & $0$ & $1$ & $0$ & $(\{\infty,0\}, 2), (\{1\},2)$ & 
$\left(\begin{smallmatrix} 0 &  -1\\ 1& 0\end{smallmatrix}\right), \left(\begin{smallmatrix}2 &  -1 \\ 1 & 0\end{smallmatrix}\right)$ & Yes\\
    \hline
$4$ &  \raisebox{-.5\height}{\includegraphics[width=50pt,height=30pt]{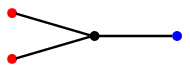}} & $\{\infty, 0,1,\infty\}$ & $ (2,3), (2,4,3), (1,2,4,3)$ & $0$ & $2$ & $1$ & $(\{\infty,0,1\}, 4)$ & 
 $\left(\begin{smallmatrix} -1 &  -1\\ 1& 0\end{smallmatrix}\right), \left(\begin{smallmatrix}1 &  -1 \\ 2 & -1\end{smallmatrix}\right), \left(\begin{smallmatrix}1 &  -2 \\ 1 & -1\end{smallmatrix}\right)$ & Yes\\
\hline
$4$ &  \raisebox{-.5\height}{\includegraphics[width=50pt,height=30pt]{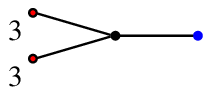}} & $\{\infty, 0,1,\infty\}$ & $ (1,2)(3,4), (2,4,3), (1,2,3)$ & $0$ & $0$ & $1$ & $(\{\infty,0\}, 3), (\{2\},3)$ & 
$\left(\begin{smallmatrix} -1 &  -1\\ 1& 0\end{smallmatrix}\right), \left(\begin{smallmatrix}2 &  -1 \\ 1 & 0\end{smallmatrix}\right)$ & Yes\\
    \hline
$5$ &  \raisebox{-.5\height}{\includegraphics[width=50pt,height=30pt]{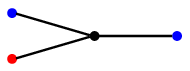}} & $\{\infty, 0,1,\infty\}$ & $ (1,2)(3,4), (2,5,4), (1,2,5,3,4)$ & $0$ & $1$ & $2$ & $(\{\infty,0,1\}, 5)$ & 
$\left(\begin{smallmatrix} -1 &  -1\\ 1& 0\end{smallmatrix}\right), \left(\begin{smallmatrix}1 &  -1 \\ 3 & -2\end{smallmatrix}\right), \left(\begin{smallmatrix}1 &  -2 \\ 1 & -1\end{smallmatrix}\right)$ & Yes\\    
\hline
$6$ &  \raisebox{-.5\height}{\includegraphics[width=50pt,height=30pt]{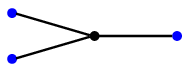}} & $\{\infty, 0,1,\infty\}$ & $ (1,2)(3,4)(56), (2,6,4), (1,2,5,6,3,4)$ & $0$ & $0$ & $3$ & $(\{\infty,0,1\}, 6)$ & 
$\left(\begin{smallmatrix} -1 &  -1\\ 1& 0\end{smallmatrix}\right), \left(\begin{smallmatrix}1 &  -1 \\ 3 & -2\end{smallmatrix}\right), \left(\begin{smallmatrix}1 &  -3 \\ 1 & -2\end{smallmatrix}\right)$ & Yes\\    
    \hline
 $6$ &  \raisebox{-.5\height}{\includegraphics[width=50pt,height=30pt]{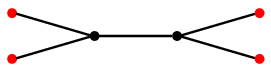}} & $\{\infty, 0,1,2,\infty\}$ & $ (1,5), (1,5,3)(2,6,4), (1,2,6,4,5,3)$ & $0$ & $4$ & $0$ & $(\{\infty,0,1,2\}, 6)$ & 
 \begin{tabular}{@{}c@{}}$\left(\begin{smallmatrix} 0 &  -1\\ 1& 0\end{smallmatrix}\right), \left(\begin{smallmatrix}1 &  -1 \\ 2 & -1\end{smallmatrix}\right),$\\ $\left(\begin{smallmatrix}3 &  -5 \\ 2 & -3\end{smallmatrix}\right), \left(\begin{smallmatrix}2 &  -5 \\ 1 & -2\end{smallmatrix}\right)$\end{tabular} & Yes\\
    \hline
     $6$ &  \raisebox{-.5\height}{\includegraphics[width=50pt,height=30pt]{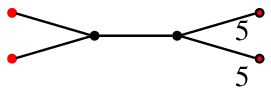}} & $\{\infty, 0,1,2,\infty\}$ & $ (2,5)(4,6), (1,5,3)(2,6,4), (1,2,4,5,3)$ & $0$ & $2$ & $0$ & $(\{\infty,0,1,\}, 5),(\{2\},5)$ & 
\begin{tabular}{@{}c@{}}$\left(\begin{smallmatrix} 0 &  -1\\ 1& 0\end{smallmatrix}\right), \left(\begin{smallmatrix}1 &  -1 \\ 2 & -1\end{smallmatrix}\right),$\\ $\left(\begin{smallmatrix}3 &  -4 \\ 1 & -1\end{smallmatrix}\right)$\end{tabular} & Yes\\    
    \hline
     $6$ &  \raisebox{-.5\height}{\includegraphics[width=50pt,height=30pt]{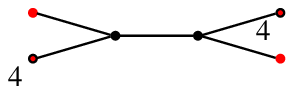}} & $\{\infty, 0,1,2,\infty\}$ & $ (2,5)(3,6), (1,5,3)(2,6,4), (1,2,3)(4,5,6)$ & $0$ & $2$ & $0$ & $(\{\infty,0,\}, 3),(\{1,2\},3)$ & 
\begin{tabular}{@{}c@{}}$\left(\begin{smallmatrix} 0 &  -1\\ 1& 0\end{smallmatrix}\right), \left(\begin{smallmatrix}3 &  -1 \\ 1 & 0\end{smallmatrix}\right),$\\ $\left(\begin{smallmatrix}3 &  -5 \\ 2 & -3\end{smallmatrix}\right), \left(\begin{smallmatrix}2 &  -5 \\ 1 & -2\end{smallmatrix}\right)$\end{tabular} & Yes\\    
    \hline
     $6$ &  \raisebox{-.5\height}{\includegraphics[width=50pt,height=30pt]{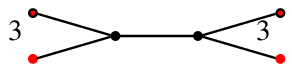}} & $\{\infty, 0,1,2,\infty\}$ & $ (1,6)(2,5), (1,5,3)(2,6,4), (1,2)(3,6,4,5)$ & $0$ & $2$ & $0$ & $(\{\infty\}, 2)(\{0,1,2\},4)$ & 
\begin{tabular}{@{}c@{}}$\left(\begin{smallmatrix} 1 &  2\\ 0& 1\end{smallmatrix}\right), \left(\begin{smallmatrix}1 &  -1 \\ 2 & -1\end{smallmatrix}\right),$\\ $\left(\begin{smallmatrix}3 &  -5 \\ 2 & -3\end{smallmatrix}\right), \left(\begin{smallmatrix}2 &  -5 \\ 1 & -2\end{smallmatrix}\right)$\end{tabular} & Yes\\    
    \hline
     $6$ &  \raisebox{-.5\height}{\includegraphics[width=50pt,height=30pt]{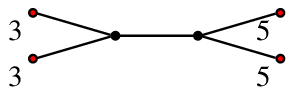}} & $\{\infty, 0,1,2,\infty\}$ & $ (1,3)(2,5)(4,6), (1,5,3)(2,6,4), (1,2,4,5)$ & $0$ & $0$ & $0$ &  \begin{tabular}{@{}c@{}}$(\{\infty,1\}, 4)(\{0\},1)$\\$(\{2\},4)$\end{tabular} & 
$\left(\begin{smallmatrix} 1 &  0\\ 1& 1\end{smallmatrix}\right), \left(\begin{smallmatrix}3 &  -4 \\ 1 & -1\end{smallmatrix}\right)$ & Yes\\    
    \hline
     $6$ &  \raisebox{-.5\height}{\includegraphics[width=50pt,height=30pt]{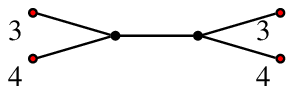}} & $\{\infty, 0,1,2,\infty\}$ & $ (1,6)(2,5)(3,4), (1,5,3)(2,6,4), (1,2)(3,6)(4,5)$ & $0$ & $0$ & $0$ & \begin{tabular}{@{}c@{}}$(\{\infty\}, 2)(\{0,2\},2)$\\$(\{1\},2)$\end{tabular} & 
$\left(\begin{smallmatrix} 1 &  2\\ 0& 1\end{smallmatrix}\right), \left(\begin{smallmatrix}3 &  -2 \\ 2 & -1\end{smallmatrix}\right)$ & Yes\\    
\hline
   $6$ &  \raisebox{-.5\height}{\includegraphics[width=50pt,height=30pt]{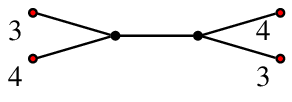}} & $\{\infty, 0,1,2,\infty\}$ & $ (1,4)(2,5)(3,6), (1,5,3)(2,6,4), (1,2,3,4,5,6)$ & $1$ & $0$ & $0$ & $(\{\infty,0,1,2\}, 6)$& 
$\left(\begin{smallmatrix} 2 &  1\\ 1& 1\end{smallmatrix}\right), \left(\begin{smallmatrix}3 &  -1 \\ 1 & 0\end{smallmatrix}\right)$ & Yes\\    
    \hline
     $7$ &  \raisebox{-.5\height}{\includegraphics[width=50pt,height=30pt]{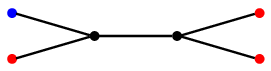}} & $\{\infty, 0,1,2,\infty\}$ & $ (1,2)(3,6), (2,6,4)(3,7,5), (1,2,3,7,5,6,4)$ & $0$ & $3$ & $1$ & $(\{\infty,0,1,2\}, 7)$ & 
\begin{tabular}{@{}c@{}}$\left(\begin{smallmatrix} -1 &  -1\\ 1& 0\end{smallmatrix}\right), \left(\begin{smallmatrix}1 &  -1 \\ 2 & -1\end{smallmatrix}\right),$\\ $\left(\begin{smallmatrix}3 &  -5 \\ 2 & -3\end{smallmatrix}\right), \left(\begin{smallmatrix}2 &  -5 \\ 1 & -2\end{smallmatrix}\right)$\end{tabular} & Yes\\    
    \hline
 $7$ &  \raisebox{-.5\height}{\includegraphics[width=50pt,height=30pt]{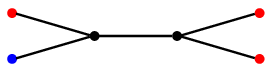}} & $\{\infty, 0,1,2,\infty\}$ & $ (2,6)(3,4), (1,6,4)(2,7,5), (1,2,7,5,6,3,4)$ & $0$ & $3$ & $1$ & $(\{\infty,0,1,2\}, 7)$ & 
\begin{tabular}{@{}c@{}}$\left(\begin{smallmatrix} 0 &  -1\\ 1& 0\end{smallmatrix}\right), \left(\begin{smallmatrix}1 &  -1 \\ 3 & -2\end{smallmatrix}\right),$\\ $\left(\begin{smallmatrix}3 &  -5 \\ 2 & -3\end{smallmatrix}\right), \left(\begin{smallmatrix}2 &  -5 \\ 1 & -2\end{smallmatrix}\right)$\end{tabular} & Yes\\    
    \hline
 $7$ &  \raisebox{-.5\height}{\includegraphics[width=50pt,height=30pt]{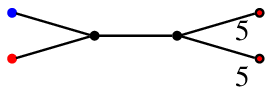}} & $\{\infty, 0,1,2,\infty\}$ & $ (1,2)(3,6)(5,7), (2,6,4)(3,7,5), (1,2,3,5,6,4)$ & $0$ & $1$ & $1$ & $(\{\infty,0,1,\}, 6)(\{2\},6)$ & 
\begin{tabular}{@{}c@{}}$\left(\begin{smallmatrix} -1 &  -1\\ 1& 0\end{smallmatrix}\right), \left(\begin{smallmatrix}1 &  -1 \\ 2 & -1\end{smallmatrix}\right),$\\ $\left(\begin{smallmatrix}3 &  -4 \\ 1 & -1\end{smallmatrix}\right)$\end{tabular} & No\\    
    \hline
 $7$ &  \raisebox{-.5\height}{\includegraphics[width=50pt,height=30pt]{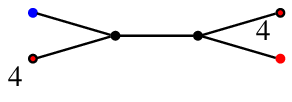}} & $\{\infty, 0,1,2,\infty\}$ & $ (1,2)(3,6)(4,7), (2,6,4)(3,7,5), (1,2,3,4)(5,6,7)$ & $0$ & $1$ & $1$ & $(\{\infty,0,\}, 4)(\{1,2\},4)$ & 
\begin{tabular}{@{}c@{}}$\left(\begin{smallmatrix} -1 &  -1\\ 1& 0\end{smallmatrix}\right), \left(\begin{smallmatrix}3 &  -1 \\ 1 & 0\end{smallmatrix}\right),$\\ $\left(\begin{smallmatrix}3 &  -5 \\ 2 & -3\end{smallmatrix}\right)$\end{tabular} & No\\    
    \hline
 $7$ &  \raisebox{-.5\height}{\includegraphics[width=50pt,height=30pt]{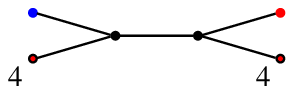}} & $\{\infty, 0,1,2,\infty\}$ & $ (1,2)(3,6)(4,5), (2,6,4)(3,7,5), (1,2,3,7,4)(5,6)$ & $0$ & $1$ & $1$ & $(\{\infty,0,2\}, 5)(\{1\},5)$ & 
\begin{tabular}{@{}c@{}}$\left(\begin{smallmatrix} -1 &  -1\\ 1& 0\end{smallmatrix}\right), \left(\begin{smallmatrix}3 &  -2 \\ 2 & -1\end{smallmatrix}\right),$\\ $\left(\begin{smallmatrix}2 &  -5 \\ 1 & -2\end{smallmatrix}\right)$\end{tabular} & No\\    
    \hline
 $7$ &  \raisebox{-.5\height}{\includegraphics[width=50pt,height=30pt]{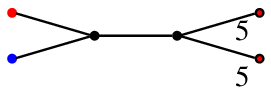}} & $\{\infty, 0,1,2,\infty\}$ & $ (2,6)(3,4)(5,7), (1,6,4)(2,7,5), (1,2,5,6,3,4)$ & $0$ & $1$ & $1$ & $(\{\infty,0,1,\}, 6)(\{2\},6)$ & 
\begin{tabular}{@{}c@{}}$\left(\begin{smallmatrix} 0 &  -1\\ 1& 0\end{smallmatrix}\right), \left(\begin{smallmatrix}1 &  -1 \\ 3 & -2\end{smallmatrix}\right),$\\ $\left(\begin{smallmatrix}3 &  -4 \\ 1 & -1\end{smallmatrix}\right)$\end{tabular} & No\\    
  \end{tabular}
  }
\caption{\label{table:subgroups}Table of Subgroups up to index $7$ with all the computed information.}
\end{table}

\bibliographystyle{abbrv}
\bibliography{bibliografia.bib} 

\begin{thebibliography}{10}

\bibitem{MR0337781}
A.~O.~L. Atkin and H.~P.~F. Swinnerton-Dyer.
\newblock Modular forms on noncongruence subgroups.
\newblock In {\em Combinatorics ({P}roc. {S}ympos. {P}ure {M}ath., {V}ol.
  {XIX}, {U}niv. {C}alifornia, {L}os {A}ngeles, {C}alif., 1968),}, pages 1--25.
  ,, 1971.

\bibitem{MR1302789}
G.~Berger.
\newblock Hecke operators on noncongruence subgroups.
\newblock {\em C. R. Acad. Sci. Paris S\'{e}r. I Math.}, 319(9):915--919, 1994.

\bibitem{2207.13365}
D.~Berghaus, H.~Monien, and D.~Radchenko.
\newblock On the computation of modular forms on noncongruence subgroups, 2022.

\bibitem{2301.02135}
D.~Berghaus, H.~Monien, and D.~Radchenko.
\newblock A database of modular forms on noncongruence subgroups, 2023.

\bibitem{MR2735087}
A.~Dooms, E.~Jespers, and A.~Konovalov.
\newblock From {F}arey symbols to generators for subgroups of finite index in
  integral group rings of finite groups.
\newblock {\em J. K-Theory}, 6(2):263--283, 2010.

\bibitem{GAP4}
The GAP~Group.
\newblock {\em {GAP -- Groups, Algorithms, and Programming, Version 4.12.2}},
  2022.

\bibitem{MR0414669}
M.~Hall, Jr.
\newblock {\em The theory of groups}.
\newblock Chelsea Publishing Co., New York, 1976.
\newblock Reprinting of the 1968 edition.

\bibitem{HSUalgorithm}
T.~Hsu.
\newblock Identifying congruence subgroups of the modular group.
\newblock {\em Proc. Amer. Math. Soc.}, 124(5):1351--1359, 1996.

\bibitem{Kulkarni1991}
R.~S. Kulkarni.
\newblock An arithmetic-geometric method in the study of the subgroups of the
  modular group.
\newblock {\em Amer. J. Math.}, 113(6):1053--1133, 1991.

\bibitem{Kurth2007}
C.~A. Kurth and L.~Long.
\newblock Computations with finite index subgroups of {${\rm PSL}_2(\Bbb Z)$}
  using {F}arey symbols.
\newblock In {\em Advances in algebra and combinatorics}, pages 225--242. World
  Sci. Publ., Hackensack, NJ, 2008.

\bibitem{Lang1995}
M.~L. Lang, C.-H. Lim, and S.~P. Tan.
\newblock An algorithm for determining if a subgroup of the modular group is
  congruence.
\newblock {\em Journal of The London Mathematical Society-second Series},
  51:491--502, 1995.

\bibitem{Millington}
M.~H. Millington.
\newblock Subgroups of the classical modular group.
\newblock {\em J. London Math. Soc. (2)}, 1:351--357, 1969.

\bibitem{MR466047}
M.~Newman.
\newblock Asymptotic formulas related to free products of cyclic groups.
\newblock {\em Math. Comp.}, 30(136):838--846, 1976.

\bibitem{MR1291394}
G.~Shimura.
\newblock {\em Introduction to the arithmetic theory of automorphic functions},
  volume~11 of {\em Publications of the Mathematical Society of Japan}.
\newblock Princeton University Press, Princeton, NJ, 1994.
\newblock Reprint of the 1971 original, Kan\^{o} Memorial Lectures, 1.

\bibitem{Stromberg2019}
F.~Str\"{o}mberg.
\newblock Noncongruence subgroups and {M}aass waveforms.
\newblock {\em J. Number Theory}, 199:436--493, 2019.

\bibitem{math/0702223}
S.~A. Vidal.
\newblock Sur la classification et le denombrement des sous-groupes du groupe
  modulaire et de leurs classes de conjugaison, 2007.

\bibitem{Wohlfahrt1964}
K.~Wohlfahrt.
\newblock An extension of {F}. {K}lein's level concept.
\newblock {\em Illinois J. Math.}, 8:529--535, 1964.

\end{thebibliography}

\end{document}